\documentclass[12pt]{amsart}
\usepackage{tikz}
\usepackage{tikz-cd}
\usepackage{pinlabel}

\usepackage{comment}

\usepackage{mathrsfs}
\usepackage{amsmath,amsthm,amsfonts,graphicx,amssymb,amscd,dsfont,euscript,enumerate,verbatim,calc,mathtools}
\allowdisplaybreaks

\usepackage{hyperref}

\usepackage{color}

\newtheorem{thm}{Theorem}[section]

\newtheorem{defn}[thm]{Definition}

\newtheorem{cor}[thm]{Corollary}
\newtheorem{lem}[thm]{Lemma}
\newtheorem{prop}[thm]{Proposition}
\newtheorem{rmk}{Remark}
\numberwithin{equation}{section}

\newcommand{\pslc}{{\mathrm{PSL}_2 (\mathbb{C})}}

\newcommand{\cp}{\mathbb{C}\mathrm{P}^1}
\newcommand{\C}{{\mathbb C}}

\newcommand{\pslnc}{\mathrm{PSL}_n(\mathbb{C})}

\newcommand{\HH}{\mathbb{H}}

\title[ On harmonic maps from $\mathbb{C}$ to $\mathbb{H}^3$] {On harmonic maps from\\ the complex plane to hyperbolic 3-space}

\author{Subhojoy Gupta and Gobinda Sau }

\address{Department of Mathematics, Indian Institute of Science, Bangalore 560012, India}
\email{subhojoy@iisc.ac.in}
\address {
Stat-Math Unit,
Indian Statistical Institute 
Bangalore,
560059, India}
\email{gobindasau@isibang.ac.in}


\thanks{ }

\DeclareMathOperator{\hp}{Hopf}

\usepackage{geometry}
 \newgeometry{vmargin={25mm, 25mm}, hmargin={25mm,25mm}}
 \usepackage[T1]{fontenc}            
\usepackage{palatino}               
\usepackage{microtype}
\begin{document}

\maketitle
\begin{abstract}
For any twisted ideal polygon in $\mathbb{H}^3$, we construct a harmonic map from $\mathbb{C}$ to $\mathbb{H}^3$ with a polynomial Hopf differential, that is asymptotic to the given polygon, and is a bounded distance from a pleated plane. Our proof uses the harmonic map heat flow. We also show that such a harmonic map is unique once we prescribe the principal part of its Hopf differential.
\end{abstract}
\section{Introduction}


%
The study of harmonic maps from the complex plane $\mathbb{C}$ to the hyperbolic plane $\mathbb{H}^2$ have received considerable attention. One of the results that has been a motivation for this article, is the following (from \cite{HTTW95}, see also \cite{HanRemarks}):

\begin{thm}[Han-Tam-Treibergs-Wan]\label{httw}
Given a polynomial quadratic differential $\phi$ on $\mathbb{C}$, there exists a harmonic map $\mathbb{C}$ to $\mathbb{H}^2$ with Hopf differential $\phi$ with image bounded by an ideal polygon. Conversely, given an ideal polygon in $\mathbb{H}^2$, there exists a harmonic map from $\mathbb{C}$ to $\mathbb{H}^2$ that is a diffeomorphism to the region bounded by that polygon.
\end{thm}

Here, recall that the Hopf differential of a map is the $(2,0)$-part of the pullback of the metric in the target space; for a harmonic map from a surface it defines a holomorphic quadratic differential in the domain. Indeed, the harmonicity of a map $u:\C \to \HH^2$ can be shown to be equivalent to the elliptic PDE 
\begin{equation}\label{vortex}
\Delta w=e^{2w}-|{\phi}|^2e^{-2w},
\end{equation}
where $w=\ln||\partial u||$ (the logarithm of the holomorphic energy density), and $\phi$ is the Hopf differential.  It turns out that \eqref{vortex} is exactly the Gauss-Codazzi equation for a space-like constant mean-curvature surface in Minkowski $3$-space, and its Gauss map is the harmonic map $u$. (See \cite{noteGobinda} for an exposition.)


\vspace{.05in} 

This article concerns harmonic maps from $\mathbb{C}$ to hyperbolic $3$-space $\mathbb{H}^3$ with polynomial Hopf differential, which can no longer be derived from solutions of the preceding equation. However, adapting the work in \cite{Min92} and \cite {HTTW95}, one can still use \eqref{vortex} to show that such a harmonic map is asymptotic to a \textit{twisted ideal polygon} in $\mathbb{H}^3$ (see Proposition \ref{myprop}). 

Here, a twisted ideal polygon  in $\mathbb{H}^3$ comprises a cyclically ordered set of ideal points in $\partial_\infty \mathbb{H}^3$ and bi-infinite geodesics between successive points (see Definition \ref{defn:tip} for a more precise definition, and see Figure \ref{fig2}).  Moreover, a map from $\C$ to $\HH^3$ is  said to be \textit{asymptotic} to a twisted ideal polygon $P$ if for any diverging sequence $\{x_k\}_{k\geq 1}$ in $\mathbb{C}$, their images under the map converge (after passing to a subsequence) to a point in $P$.

\vspace{.05in}

We prove that indeed, any  twisted ideal polygon arises as the asymptotic limit of some such harmonic map:

\begin{thm}\label{main0}
Given a twisted ideal polygon in $\HH^3$ with $n\geq 3$ ideal vertices, there exists a harmonic map from $\mathbb{C}$ to $\mathbb{H}^3$ asymptotic to that polygon, and has a Hopf differential $q(z) dz^2$ where $q(z)$ is a polynomial of degree $(n-2)$.
\end{thm}


\vspace{.05in} 

Our proof uses the harmonic map heat flow, originally introduced in a seminal paper of Eells-Sampson (\cite{EellsSampson}) in the context of \textit{compact} Riemannian manifolds.  Starting with an initial map $u_0: \mathbb{C} \to \HH^3$ we consider the PDE
\begin{equation}\label{heat1}
\begin{split}
\frac{\partial u}{\partial t}&=\tau(u(x,t))\\
 u(x,0)&=u_0(x),  
\end{split}
\end{equation}
where $\tau(u(x,t))$ is the \textit{tension field} of $u_t(x):=u(x,t)$ (see \S \ref{tf} for the definition); this can be thought of as the gradient flow for the energy functional. 
Indeed, Eells-Sampson showed that for compact Riemannian manifolds and when the target is non-positively curved, the solution $u(x,t)$  of the above equation exists for all time, and converges to a harmonic map as $t\rightarrow\infty$. 

This method has also been studied in the context of non-compact Riemannian manifolds, notably by Li-Tam \cite{LiTam}, Jiaping Wang \cite{JiapingWang},  and Meng Wang \cite{MengWang},  amongst others. We shall use their results on the long-time existence of the harmonic map heat flow.  However the  problem of whether the solution \textit{converges} to a harmonic map is non-trivial; indeed, in \S2.1 we provide an example where it does not.

\vspace{.05in}

We shall in fact prove the following result, from which Theorem \ref{main0} is an immediate corollary:

\begin{thm}\label{main}
Given a twisted ideal polygon $P$ in $\HH^3$, there is a choice of an initial map $u_0:\C \to \HH^3$ such that the harmonic map heat flow \eqref{heat1} has a solution for all $t>0$. Moreover, for each $t\geq 0$, the map $u_t:\C \to \HH^3$ 
\begin{itemize}
    \item has a tension field with exponentially decaying norm (Lemmas \ref{intptf} and \ref{uboundtf}), 
    \item is asymptotic to $P$ (Corollary \ref{tend0}), 
    \item is trapped in a fixed neighborhood of the convex hull of the ideal vertices of $P$ (Lemma \ref{chs}), and 
    \item is a uniformly bounded distance from a pleated plane asymptotic to $P$ (Lemma \ref{Pb} and Corollary \ref{Psib}), where the pleated plane is defined in Definition \ref{mapP}.
\end{itemize}
Finally, as $t\to \infty$, the maps $u_t$ converge uniformly on compact sets to a harmonic map $u_\infty:\C \to \mathbb{H}^3$ with a polynomial Hopf differential, that is also asymptotic to $P$.
\end{thm}


\vspace{.05in}

Our construction of the initial map $u_0$ proceeds by modifying a harmonic map  $h$ to a planar polygon (in $\HH^2$), the existence of which follows from Theorem \ref{httw}. A key step to prove the convergence of the flow is to establish a uniform distance bound of $u_t$ from a \textit{pleated plane} map $\Xi:\mathbb{C} \to \HH^3$ that, although not $C^1$-smooth, is \textit{piecewise}-harmonic.


\vspace{.05in} 

In the final section, we characterize the non-uniqueness of the maps obtained in Theorem \ref{main0}. From our construction,  we observe that there are in fact \textit{infinitely many} harmonic maps asymptotic to the same twisted ideal polygon. We shall prove (Proposition \ref{prop:princ}) 
that there is a unique such harmonic map if we, in addition, prescribe the ``principal part" of its Hopf differential (see Definition \ref{defn:princ}).

\vspace{.05in}

In a forthcoming article \cite{GS2}, we address the question of existence (and uniqueness)  of \textit{equivariant} harmonic maps from $\HH^2$ to $\HH^3$ asymptotic to a given framing, where the equivariance is with respect to a  framed representation from a (punctured-) surface-group to $\pslc$. 
\vspace{.05in}

It would be interesting to extend Theorem \ref{main} to the case when $q$ in  the Hopf differential $q(z)dz^2$ is a more general entire function; in particular, one can ask:

\vspace{.05in}

\textbf{\textit{Question.}} Given a quasicircle $\Lambda \subset \partial_\infty \mathbb{H}^3$, does there exist a harmonic map $h:\C \to \HH^3$ that is asymptotic to $\Lambda$? 

\vspace{.05in}

(Note that for the special case that $\Lambda$ is a round circle, there is such a map, as shown in \cite{CollRos} where they disproved a conjecture of Schoen (\cite{Schoen}).)

\vspace{.1in}

The arguments in this paper also extend to the case of harmonic maps from the complex plane into hyperbolic $n$-space $\HH^n$ for $n>3$. It would also be interesting to explore analogues of these results when the target is replaced by other non-compact symmetric spaces, particularly those of higher rank; we hope to pursue that in future work. In the case that the target is the symmetric space of $\pslnc$ for $n>2$, the existence of such harmonic maps is already studied in the context of  ``wild non-abelian Hodge theory" (see, for example, \cite{BiqBoalch} and \cite{Mochi}, or the more recent \cite{LiMoch}), but we are not aware of a geometric study of their asymptotic behavior.


\vspace{.1in}

\subsection*{Acknowledgements.} Several parts of the work in this paper and its sequel \cite{GS2} are contained in the PhD thesis of GS \cite{thesisGobinda}, written under SG's supervision. Both authors would like to thank Qiongling Li for her help and advice. SG thanks Mike Wolf for numerous conversations about harmonic maps, and Yair Minsky for providing lasting inspiration. This work was supported by the Department of Science and Technology, Govt.of India grant no. CRG/2022/001822, and by the DST FIST program - 2021 [TPN - 700661].

\section{Preliminaries}\label{3}
In this section we recall some of the basic notions we shall need in the rest of the paper.
\subsection{Harmonic maps and the heat flow method}\label{tf}

We provide a general discussion in the context of maps between two Riemannian manifolds $(M,g)$ and $(N,h)$; a reference for this is \cite{Nishikawa02}. 
We shall subsequently specialize to the case when $M=\mathbb{C}$, equipped with a conformal metric that we shall describe, and $N = \HH^3$ equipped with the hyperbolic metric.

\begin{defn}[Harmonic map]\label{defn:harm} A $C^2$-smooth map $u:(M,g)\rightarrow (N,h)$ is called harmonic if it is a critical point  of the energy functional
\[E^U(u)=\frac{1}{2}\int_{U}||du||^2\]
on every relatively compact open subset $U$ of the domain. Here, the function in the integrand is the called the {energy density} of $u$, and is denoted by $e(u)$.
\end{defn}

An alternative definition of harmonicity is in terms of the tension field:

\begin{defn}[Tension field] The tension field of a map $u:(M,g)\rightarrow (N,h)$ is defined to be 
\[\tau(u)=\text{Tr}_g(\nabla du).\] In local coordinates, components of the tension field are given by\[\tau(u)^{\alpha}=\Delta_gu^{\alpha}+g^{ij}(x)\frac{\partial u^{\beta}}{\partial x^i}\frac{\partial u^{\gamma}}{\partial x^j}\Gamma'^{\alpha}_{\beta \gamma}(u(x)),\] where $\Delta_g$ is the Laplace-Beltrami operator on $(M,g)$ and $\Gamma'^{\alpha}_{\beta \gamma}$ are the Christoffel symbols for the metric $h$ on $N$.
\end{defn}

 \begin{lem}
     A map  $u:(M,g)\rightarrow (N,h)$ is harmonic if and only if its tension field $\tau(u) =0$.
 \end{lem}
 \begin{proof}
      By the first variation  of the energy functional, for a $1$-parameter family of maps $u_t$ starting from $u_0 =u$ defined by a variation vector field $V$, we have 
  \[\frac{d}{dt}\biggr\rvert_{t=0}E^U(u_t)=-\int_{U}\langle V,\tau(u)\rangle dvol_g.\]
 where  $\tau(u)$ is the {tension field} of $u$.
 Since $V$ is arbitrary, the statement of the lemma follows. 
 \end{proof}

The harmonic map heat flow was already mentioned in \S 1 -- see \eqref{heat1} --  and can be thought of as the gradient flow for the energy functional. Indeed, in \cite{EellsSampson} Eells-Sampson showed that in the case that the manifolds $M,N$ are compact, the flow exists for all time and converges to a harmonic map.

In our case when the manifolds are non-compact, we shall use the following existence result of J. Wang in \cite[Theorem 3.1]{JiapingWang} :

 \begin{thm}[Long-time solution]\label{wang}
Let $M$ and $N$ be two complete Riemannian manifolds such that the sectional curvatures $K_N\leq0$. Let $u_0:M \rightarrow N $ be a $C^2$ map. If $$\left(\int_MH(x,y,t)|\tau(u_0)|^2(y)dy\right)^\frac{1}{2}=b(x,t)$$ is finite for all $(x,t)\in M\times(0,\infty)$ where $H(x,y,t)$ is the heat kernel of $M$, then $\eqref{heat1}$ has a long time solution $u(x,t)$ defined for all $(x,t)\in M\times(0,\infty)$, that satisfies the tension field bound $|\tau(u)(x,t)|\leq b(x,t)$. Moreover, if $N$ is simply-connected, and for any $T>0$, the integral  $\int_{0}^{T}\int_\mathbb{C} e^{-cr^2(x)}b^2(x,t)dxdt<\infty$ for some $c>0$, the solution is unique.  
\end{thm}

The above result shall apply to our setting when $M = \C$ and $N = \HH^3$. We remark that even if a solution exists, it is not always true that the solution will converge to a limiting harmonic map (see the following example). This is in contrast to the case when the domain manifold $M$ has positive lower bound of the spectrum $\lambda(M)>0$ (see \cite[Theorem 5.2]{LiTam}). 

\subsubsection*{An example}\label{lateruse}
Here we give an example of an initial map $u_0 = u:\C \to \HH^3$ such that the harmonic map heat flow \eqref{heat1} does \textit{not} converge. 

Namely, let $u(x,y)=(x,y,t_0)$ for some $t_0>0$, where we have used the upper half-space model of hyperbolic $3$-space where $\HH^3 \cong \mathbb{C} \times \mathbb{R}^+$.

If we assume  that the solution of \eqref{heat1} is of the form $u(x,y,t) = (x,y, f(t))$, the harmonic map heat flow reduces to the ODE
\[\frac{d f}{d t}=\frac{2}{f}.\] Using the initial condition $f(0)=0$ we obtain $f(t)=\sqrt{4t+t_0^2}$ and consequently, $u(x,y,t)=(x,y,\sqrt{4t+t_0^2})$ is a solution to the harmonic map heat flow, with initial map $u$. Clearly, as $t\rightarrow\infty$, solution $u(x,t)$ does not converge. We remark that one can compute that in this case the tension field is uniformly bounded, so the hypotheses of Theorem \ref{wang} hold.

\subsection{Harmonic maps from $\mathbb{C}$ to $\mathbb{H}^2$}\label{prelim3}
In this subsection, we recall some of the previous work that relates the asymptotic behaviour of harmonic maps from $\mathbb{C}$ to $\mathbb{H}^2$, to the horizontal and vertical foliations of the Hopf differential. We shall use some of these estimates in our arguments. Throughout this subsection, the domain $\Sigma$ will be a (possibly non-compact) Riemann surface of finite type. 

\begin{defn}[Hopf differential] 
 For a $C^2$-smooth map $u:\Sigma\rightarrow (N,h)$, the \emph{Hopf differential} of $u$ denoted by $\hp(u)$ is \[\phi=(u^*h)^{2,0},\]
a quadratic differential on $\Sigma$ locally of the form $\phi(z)dz^2$. 
\end{defn}
 
 \noindent \textit{Remark.} It is well-known that if $u$ is harmonic, then  $\phi$ is a holomorphic quadratic differential on $\Sigma$ (\text{see for example} \cite[Theorem 10.1.1, p-577]{JostRG17}).   

\vspace{.05in} 

The following notion from \cite[Definition 2.5]{Gupta21} (see also \S 2.3 of \cite{GupWolf}) will be used at times in this paper, especially in the final subsection \S3.7:

\begin{defn}[Principal part]\label{defn:princ}
The principal part of a meromorphic quadratic differential $q$ at a pole is a meromorphic $1$-form $\omega$ defined in a neighborhood $U$ of the pole such that $\sqrt q - \omega$ is integrable on $U$. In local coordinates, if $U \cong \mathbb{D}^\ast$ and $q$ has a pole at $0$, then $\omega = z^{-n/2}P(z)dz$ where $P(z)$ is a certain polynomial of degree $\lfloor \frac{n-2}{2} \rfloor$ comprising terms in the Laurent expansion of $\sqrt q$ -- see equations (4) and (5) in \cite{GupWolf}.
\end{defn}

\noindent\textit{Remark/Notation.} In this paper, we shall consider holomorphic quadratic differentials on $\C$ arising as Hopf differentials of harmonic maps with domain $\C$; such a differential has a single pole at $\infty$, and we shall just refer to the principal part there as  the ``principal part of the Hopf differential".  

\vspace{.05in}

\begin{defn}[Horizontal and vertical foliations] Let $\phi$ be a holomorphic quadratic differential on $\Sigma$. Recall that each $p\in \Sigma$, $\phi$ defines a map $\phi_p:T_p\Sigma\rightarrow \mathbb{C}$ satisfying $\phi_p(\lambda v)=\lambda^2\phi_p(v)$ for any $v\in T\Sigma$ and  $\lambda\in \mathbb{C}$.
A tangent vector $v\in T_p\Sigma$ is called horizontal (respectively, vertical) for the quadratic differential $\phi$ if $\phi_p(v)>0$ (respectively, $\phi_p(v)<0$). The set of horizontal or vertical tangent vectors in $T\Sigma$ forms a smooth line field away from the set $F$ of zeros of $\phi$; this can be integrated to define the horizontal and vertical foliations of $\Sigma \setminus F$. At any point in $F$, these foliations have prong-type singularities (see Figure 1). 

\end{defn}


\noindent  \textit{Remark.} In the case that the domain is the complex plane $\C$, a \textit{polynomial quadratic differential} $q$ is of the form $q(z)dz^2$ where $q(z)$ is a polynomial of degree $n\geq1$. Such a holomorphic quadratic differential $q$ has a pole of order $n+4$ at infinity, and there are exactly $n+2$ horizontal (or vertical) directions asymptotic to infinity.

\begin{figure}\label{hf1}
\centering
\includegraphics[scale=.5]{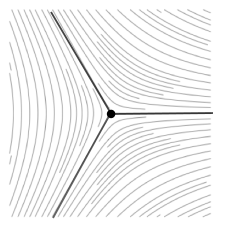} 
\caption{Horizontal foliation for $zdz^2$ on $\mathbb{C}$}
\label{hf}
\end{figure}

\begin{defn}[Quadratic differential metric]\label{defn:qmetric}
A holomorphic quadratic differential $q$ on a Riemann surface $\Sigma$ induces a conformal metric on $\Sigma$  (the $q$-metric) given by the local expression $|q(z)||dz|^2$, which is singular at the zeros of $q$. Since $q$ is holomorphic, the curvature vanishes away from these singularities; the metric is thus a singular flat metric on $\Sigma$ . 
\end{defn}

\noindent \textit{Remark.} For a polynomial Hopf differential on $\C$ as in the preceding remark, the induced singular flat metric has $(n+2)$ Euclidean half-planes isometrically embedded in a cyclic order around $\infty$. In fact, one can choose these half-planes to be \textit{horizontal}, i.e.\ foliated by horizontal lines, or \textit{vertical}, i.e.\ foliated by vertical lines. 

\vspace{.05in} 

From now on, let $h:\mathbb{C}\rightarrow\mathbb{H}^2$ be a harmonic diffeomorphism to its image with Hopf differential $\phi$. Let $\xi=x+iy$ be the canonical coordinates away from the zeros of $\phi$, where the Hopf differential has the form  $\phi = d\xi^2$. Then a short computation (see, for example \cite[Proposition A.2.1]{Huang16}) shows that:
\begin{equation}\label{hopf}
h^*(\rho)=(e+2)dx^2+(e-2)dy^2,
\end{equation} where $\rho$ is the metric on $\HH^2$ and $e$ is the energy density of $h$ with respect to the $\phi$-metric. 
\vspace{.05in}

The key analytical estimate is the following (see \cite[\S 5]{HanRemarks}, or \cite[Lemma 3.3]{Minskylengthenergy}): 

\begin{lem}\label{est0} 
In the setup above, the energy density satisfies the estimate $\lvert e (\xi) - 2\rvert = O(e^{-\alpha \lVert \xi \rVert})$ for some constant $\alpha>0$. 
\end{lem}
\begin{proof}[Sketch of the proof]
Setting $w_1=w-\frac{1}{2}\ln(|\phi|)$ in \eqref{vortex} we obtain,
\begin{equation}\label{vort2}
\Delta w_1=e^{2w_1}-e^{-2w_1}	
\end{equation}
where $\Delta$ is the Laplacian in the canonical coordinates $\xi$. The energy density is then $e = 2 \cosh{2w_1}$. 
	
By \cite[Proposition 5.1]{HanRemarks}, for any $R>1$, there is an absolute constant $C>0$ such that we have the estimate $0\leq w_1(p)\leq C$ for any point $p$ whose distance from the zeros of $\phi$ in the metric induced by $\phi$ is at least $\frac{1}{2}R$.  This uniform bound can in fact be improved to a bound that decays exponentially in $R$: consider the comparison function\[F(x,y)=\frac{C}{\cosh (R/2\sqrt{2})}\cosh(2x)\cosh(2y).\] Then $\Delta F = 4F $ and  $\Delta(F-w_1)=(4F-2\sinh 2 w_1)\leq 4(F-w_1)$, and on the boundary of $B(z,\frac{1}{2}R)$ we have $ F\geq C\geq w_1$. Applying the maximum principle we obtain $F - w_1 \geq 0$, and thus 
\begin{equation}\label{eest}
w_1\leq F =  O(e^{-\beta R}),
\end{equation} for some $\beta>0$. 
Hence we obtain the estimate 
\begin{equation}\label{ede}
e(p)=2\cosh 2 w_1(p) =  2+O(e^{ -\alpha R(p)})
\end{equation}
where $R(p)$ is the distance of the point $p$ from the zeros of $\phi$.

Moreover, since the set of zeros is a finite set, $R(p)$ equals $\lVert p \rVert $ up to a finite additive constant. The desired estimate of the energy density follows. \end{proof}


\noindent  \textit{Remarks.} (i) From the previous lemma one can derive estimates of the asymptotic behavior of the harmonic map. In particular, from \eqref{hopf}  it is immediate that far from the zeros of $\phi$, the horizontal vector $\partial/\partial x$ maps via $h_\ast$ to approximately twice its length, while the image of a vertical vector $\partial/\partial y$ is approximately of zero length.

(ii) For convenience, we shall henceforth use a scaling of the singular-flat metric induced by  the Hopf differential $\phi$ on the domain, where the scaling is by a factor $2$. This will ensure that the map is an almost-isometry in the horizontal direction, far from the zeros of $\phi$. We shall also refer to this as the $4\phi$-metric since it is the metric induced by four times the Hopf differential. 

\smallskip

In fact, one can show that far from the zeros of the Hopf differential, the harmonic map is approximated well by a map that collapses (the vertical direction) to a geodesic line $L$. We state this more precisely in the following proposition, which is implicit in \cite[\S3]{HTTW95} -- see also \cite[Theorem 4.2]{Min92}. In the statement the collapsing map $\Pi_L: \mathbb{C} \to \HH^2$ is defined as $\Pi_L = \gamma \circ \pi$ where $\pi(x,y) = x$ and $\gamma:\mathbb{R} \to L$ is a parametrized geodesic. 



\begin{prop}\label{est} Let $h:\C \to \HH^2$ be a harmonic map with a polynomial Hopf differential $\phi$ of degree $n\geq 1$. Let $H$ be a  horizontal half-plane of the singular flat metric induced by $4\phi$. Then there is a geodesic line $L \subset \mathbb{H}^2$ such that the restriction $h\vert_H$ is asymptotic to $\Pi_L$ in the following sense: Assume that $H = \{(x,y)\ \vert\ y>0\}$ and, by a post-composing with an isometry of $\HH^2$, assume that $L$ is the vertical geodesic in the upper half-plane model of $\HH^2$ , such that $\gamma(x) =(0, e^{x})$. Then if $h\vert_H(x,y) = (f(x,y), g(x,y))$ then 
\begin{itemize}
    \item[(i)] $\lVert f (x,y) \rVert_{C^1} = O(e^{-\alpha \lVert (x,y)\rVert})$

\item[(ii)] $\lVert g (x,y) - e^{x} \rVert_{C^1} = O(e^{-\alpha \lVert (x,y)\rVert})$
    
\end{itemize}
for some $\alpha>0$, where $\lVert \cdot \rVert_{C^1}$ denotes the $C^1$-norm, and $\lVert (x,y) \rVert = \sqrt{x^2 + y^2}$.
\end{prop}

\begin{proof}
We begin by observing that from \eqref{hopf} and Lemma \ref{est0} we obtain length estimates of the images of horizontal  and vertical arcs (denoted by $\gamma_h$ and $\gamma_v$ respectively) in the domain $\C$. Namely, we have:
\begin{equation}\label{inside1}
l(h(\gamma_h))=\int_{0}^{L}\sqrt{\frac{1}{4}(e+2)}dx=L+O(e^{-\alpha R})
\end{equation}
\begin{equation}\label{inside2}
l(h(\gamma_v))=\int_{0}^{L}\sqrt{\frac{1}{4}(e-2)}dy=O(Le^{-\alpha R})
\end{equation}
where there is an additional factor $1/4$ in the integrand when compared with the natural coordinates of $\phi$ (see \eqref{hopf}) because we are considering the $4\phi$-metric -- \textit{c.f.}  Remark (ii) above. 

It turns out (see \cite[Lemma 3.2]{HTTW95}) that the geodesic curvature of $h(\gamma_h)$ is given by \[k(h(\gamma_h)) = -\frac{1}{4}
(e-2)^{1/2}(e+2)^{-1}\frac{\partial e}{\partial y}\]
where there is an extra factor of $1/2$ since the natural coordinates of the $4\phi$-metric (that we continue to denote by $x$ and $y$) differs from that of the $\phi$-metric by a scaling by $2$ in each coordinate.

Using  Lemma \ref{est0} and the gradient estimate for solutions of the elliptic equation \eqref{vort2}, we then conclude that 
\begin{equation}\label{gc}
k(h(\gamma_h))=O(e^{ -\alpha R}).
\end{equation}
where $R$ is the distance of $\gamma_h$ from the origin in the $4\phi$-metric. 

It is a consequence of \cite[Lemma 3.1]{HTTW95} that an arc in $\HH^2$ with geodesic curvature less than $1$ is uniformly close to a geodesic (where the distance bound does not depend on the length of the curve.) This is also known as the Canoeing Lemma  in the hyperbolic plane (see \cite[Theorem 2.3.13]{Hubbard06}). In particular, if $\gamma_R$ is a bi-infinite horizontal line at height $R$ in the horizontal half-plane $H$, then its image under $h$ is uniformly close to a geodesic line $L$ in $\HH^2$. Note that this  geodesic line $L$ is  independent of the height $R$, since applying  \eqref{inside2} one can show that the images of  two horizontal lines at two different heights  are asymptotic to each other  (at both ends). (See \cite[Lemmas 3.3, 3.4]{HTTW95} for details.) Indeed, if $\{\gamma_R(t)\}_{t\in \mathbb{R}}$ is the parametrized horizontal line at height $R$, we have that the distance
\begin{equation}\label{hdistance}
\text{dist}(h(\gamma_R(t)),L) =  O\left(e^{-\alpha \lVert \gamma_R(t)\rVert } \right)
\end{equation} 
for some $\alpha>0$.

Recall that we can assume that, by post-composing with an isometry of $\HH^2$, that $\gamma(t)=(0,e^t)$. Thus, using the expression for $h\vert_H$ that we assumed in the statement of the Proposition, we have, \begin{equation}\label{nearer} 
d_{\mathbb{H}^2}((f(x,y) ,g(x,y)),(0,e^{x}))\leq ce^{-\alpha \lVert (x,y) \rVert}
\end{equation}
from which the $C^0$-estimates in (i) and (ii) follow. 

To complete the proof, we only need to establish exponentially decaying bounds for $\left|\left|\left|h_i{_*}\frac{\partial}{\partial y}\right|\right|-\left|\left|(\Pi{_L})_\ast\frac{\partial}{\partial y}\right|\right|\right|$ and $\left|\left|\left|h_i{_*}\frac{\partial}{\partial x}\right|\right|-\left|\left|(\Pi{_L})_\ast\frac{\partial}{\partial x}\right|\right|\right|$, where recall $\Pi_L(x,y) = (0, e^{x})$ is the collapsing map from $H$ to $L$, and thus have derivatives satisfying \[\left|\left|(\Pi{_L})_\ast\frac{\partial}{\partial x}\right|\right|=1 \quad \text{and} \quad\left|\left|(\Pi{_L})_\ast\frac{\partial}{\partial y}\right|\right|=0.\]
These bounds are an immediate consequence of \eqref{hopf} and Lemma \ref{est0}, after, once again, correcting for the scaling for the coordinates of the $4\phi$-metric, compared to those of the $\phi$-metric (\textit{c.f.} Remark (ii) following Lemma \ref{est0}). 
\end{proof}

\textit{Remark.} One can derive from the above estimates that the image of a harmonic map $h:\C \to \HH^2$ with a polynomial Hopf differential $\phi$ of degree $m\geq 1$ is an ideal polygon with $(m+2)$ geodesic sides, where each side corresponds to a horizontal half-plane of the $\phi$-metric. A similar argument applies in the case that the target is $\HH^3$, as we shall see in the next subsection, so we shall refer to that for details.

\subsection{Harmonic maps from $\C$ to $\HH^3$}

One way to obtain a harmonic map $h:\C \to \HH^3$ is to post-compose  a harmonic map from $\C$ to $\HH^2$ with an isometric embedding of $\HH^2$ in $\HH^3$. In that case the image is asymptotic to an ideal polygon that is contained in a totally-geodesic copy of $\HH^2$ in $\HH^3$ (see the Remark following Proposition \ref{est}).  In this subsection, we define the notion of a \textit{twisted} ideal polygon (that may not be contained in a totally geodesic plane), and show that, in general, the image of a harmonic map from $\C$ to $\HH^3$ with polynomial Hopf differential is asymptotic to such a twisted polygon.

\begin{defn}[Twisted ideal polygon]\label{defn:tip} Let $\{\xi_1,\xi_2,\ldots, \xi_n\}$ be $n$ points in the ideal boundary $\partial_\infty \HH^3 = \cp$, satisfying (i) there are at least three distinct points, and (ii) successive points $\xi_i$ and $\xi_{i+1}$ are distinct, for each $i$. Then a  twisted ideal $n$-gon in $\HH^3$ is a cyclically ordered set of bi-infinite geodesic lines $\{\gamma_1,\gamma_2,\ldots\gamma_n\}$ in $\HH^3$ such that $\gamma_i$ is between $\xi_i$ and $\xi_{i+1}$.
\end{defn}

\noindent \textit{Remark.} The conditions (i) and (ii) ensure that the twisted ideal polygon is ``non-degenerate".

\medskip

To prove the main result of this subsection, we shall need the following fact concerning curves with small geodesic curvature in $\HH^3$, that generalizes \cite[Lemma 3.1]{HTTW95} which concerned curves in a Hadamard surface. (See also  \cite[Theorem 2.3.13]{Hubbard06} for curves in $\HH^2$, where this is called a ``canoeing theorem".)  This is already known for curves in $\HH^n$ - see, for example, \cite[Lemma 2.5]{Leininger}. However, we provide a proof for the sake of completeness, that closely follows that of \cite[Lemma 3.1]{HTTW95}.

\begin{lem}$(\text{Canoeing lemma in}\hspace{1 mm} \mathbb{H}^3)$\label{canoe}
Let $\gamma:\mathbb{R}\rightarrow\mathbb{H}^3$ be a $C^2$-smooth curve joining two points $\xi_1, \xi_2 \in \partial_\infty \mathbb{H}^3$ with geodesic curvature $k_{\gamma}$ bounded above by $\epsilon$ where $\epsilon<1$.  Let $\sigma$ be the bi-infinite geodesic joining $\xi_1$ and $\xi_2$. Then $d_{\mathbb{H}^3}(x,\sigma)\leq C\epsilon$ for some $C>0$ independent of $\gamma$  and for all $x$ on $\gamma$.
\end{lem} 
\begin{proof}
Let $\gamma:[0, \ell] \rightarrow \mathbb{H}^{3}$ be a $C^{2}$ curve parametrized by arclength. Let $\gamma^{*}$ be the complete geodesic through $\gamma(0)$ and $\gamma(\ell)$. Without loss of generality assume that $\gamma^{*}$ is the vertical geodesic passing through $(0,0,1)$. Let $(u, v, \varphi)$ be the Fermi coordinates such that $v=0$ is the geodesic $\gamma^{*}$ and $v$ is the geodesic for the point $(u, v, \varphi)$ to $\gamma^{*}$ (taking $(0,0,1)$ as base point). Put $u=\frac{1}{2}\ln(x^2+y^2+z^2), v=\ln\frac{\sqrt{x^2+y^2+z^2}+\sqrt{x^2+y^2}}{z},\phi=\arctan(\frac{y}{x})$. In these coordinates, the metric of  $\mathbb{H}^3$ is given by \[ds^{2} = \cosh^2(v) du^{2} + dv^2+\sinh^2(v) d\phi^2.\]Let $k_{\gamma}(t)$ be the geodesic curvature of $\gamma(t)=(u(t), v(t), \varphi(t))$. From the above computations, we have,
\[
k_{\gamma}^{2}(t)=\left[\cosh^2(v)\left(\ddot{u}+\Gamma_{12}^{1} \dot{u} \dot{v}\right)^{2}+\left(\ddot{v}+\Gamma_{11}^{2} \dot{u}^{2}+\Gamma_{33}^{2} \dot{\varphi}^{2}\right)^2+\sinh^2(v)\left(\ddot{\varphi}+\Gamma_{23}^{3} \dot{v} \dot{\varphi}\right)^{2}\right]
\]
Suppose the maximum $v_{\max }$ of $v$ is attained at $t=0$ or $t=l$, then we have $v_{\max }=0$. Otherwise at some interior point $0<t_{0}<l$ where $v$ attains its maximum, $\dot{v}=0$ and $\ddot{v} \leq 0$. Let $(u(t_0), v(t_0),\phi(t_0))=(u_0, v_0,\phi_0)$. Here $v_{0}=v_{\max } \geq 0$. Since $k_{\gamma}^{2} \leq\epsilon^{2}$, at $\left(u_{0}, v_{0}\right)$,\[\left(\ddot{v}-\sinh(v) \cosh(v)\dot{u}^2-\sinh(v) \cosh(v) \dot{\varphi}^{2}\right)^{2}<\epsilon^2\] which gives $\sinh(v)\cosh(v) \left(\dot{u}^2+\dot{\varphi}^{2}\right)<\epsilon$ at $\left(u_{0}, v_{0}\right)$.
	
Since $\gamma$ is arc-length parametrized, $\left|\gamma^{\prime}\right|=1$ which gives $\cosh ^2(v) \dot{u}^{2}+\dot{v}^{2}+\sinh ^2(v) \dot{\varphi}^{2}=1$, hence at $(u_0,v_0,\phi_0)$, we have \[\frac{\sinh(v) \cosh(v)\left(\dot{u}^2+\dot{\varphi}^{2}\right)}{\cosh^2(v)\dot{u}^2+\sinh^2(v)\dot{\phi}^2}<\epsilon.\]
Now if one  put $b=\cosh(v)$ and $a=\sinh(v)$, then using the fact $b^2-a^2=1$, we have the inequality \[\frac{a}{b}\leq\frac{a b\left(\dot{u}^{2}+\dot{\varphi}^{2}\right)}{b^{2} \dot{u}^{2}+a^{2} \dot{\varphi}^{2}},\] and consequently, $\tanh(v_0)<\epsilon$. Since $\epsilon<1 $, we conclude that $v_{max} = v_0 < C\epsilon$ for some constant $C>0$. 
\end{proof}

\medskip 

We now prove the main result of this subsection: 

\begin{prop}\label{myprop}
Let $h$ be a harmonic map from $\mathbb{C}$ to $\mathbb{H}^3$ with polynomial Hopf differential $q$ of degree $m\geq 1$. Then $h$ is asymptotic to a twisted ideal polygon $P$ with $m+2$ ideal vertices. Moreover, if $q$ has degree $m=0$, then the image of $h$ is a geodesic line.
\end{prop}
\begin{proof}
In what follows we shall refer the reader to results in \cite{Min92}, which this proof crucially relies on. In the notation of that paper, we  define $\mathcal{G}$ by 
\begin{equation*}
    \sinh\mathcal{G} = \frac{\mathcal{J}}{2}
\end{equation*}
where $\mathcal{J}$ is the absolute value of the Jacobian of $h$ with respect to the $q$-metric. Note that the energy density $e(h) =2 \cosh\mathcal{G}$ (see equation (3.1) of \cite{Min92}, and compare with \eqref{ede}).

There is a version of \eqref{vortex} for the case when the target is a surface of negative sectional curvature $K\leq 0$, namely 
\begin{equation*}
    \Delta \mathcal{G} = -4 K \sinh \mathcal{G}
\end{equation*}
which is derived by a Bochner formula for $h$ (see \S1 of \cite{HTTW95}). 

As observed in \cite{Min92}, if $h$ is an immersion to $\HH^3$, i.e.\ away from the zeros of $\mathcal{G}$, the immersed surface in the image has negative sectional curvature by \cite[Theorem 8]{Sampson}, and the above equation holds. 
Moreover, far away from the zeros of the Hopf differential $q$, one also obtains an exponential decay of $\mathcal{G}$ (see \cite[Theorem 3.4]{Min92}) by an application of the maximum principle (\textit{c.f.} the sketch of the proof of Lemma \ref{est0}). In particular, we obtain the estimate \eqref{eest} even in this case.

Finally, we can show that the discussion in the proof of Proposition \ref{est} holds even in this case when the target is $\HH^3$. First, from the above discussion, the exponential decay implies that the estimates \eqref{inside1} and \eqref{inside2} hold (\textit{c.f.} equation (3.1) in \cite{Min92}). Second, Theorem 3.5 of \cite{Min92} shows that if one considers a  leaf of the horizontal foliation of  $q$ far from its zeros, the geodesic curvature of the image is small (i.e. tending to zero as the distance from the zeros increases). Finally, by Lemma \ref{canoe}, the image of such a horizontal leaf is close to a geodesic line, where the distance tends to zero the further away the horizontal leaf is from the zeros of $q$. 

Since  the horizontal foliation of $q$ comprises $(m+2)$ half-planes around $\infty$ (see the remark following Definition \ref{defn:qmetric}), we conclude that in each half-plane, the images of the horizontal lines under $h$ converge to a geodesic line in $\HH^3$.  This implies that $h$ is asymptotic to a cyclically ordered collection of geodesic lines $\gamma_1,\gamma_2,\ldots, \gamma_{m+2}$ in $\HH^3$. Moreover, a pair of horizontal lines in successive half-planes can be connected by a vertical line segment that has its length bounded by a constant, and is arbitrarily far from the zero set of $q$. By \eqref{inside2}, this implies that $l_i$ and $l_{i+1}$ have a common limiting point $\xi_i$ in the ideal boundary $\cp$. We also know that each successive points $\xi_i$ and $\xi_{i+1}$ are distinct, since they are ideal endpoints of a geodesic line in $\HH^3$. These are the same arguments as those in the proofs of Lemmas 3.3 and 3.4 of \cite{HTTW95}.

To complete the proof that this configuration of geodesic lines that $h$ is aymptotic to, is indeed a twisted ideal polygon, it remains to show property (i) in Definition \ref{defn:tip}, namely that there are at least three distinct points in the set $\{\xi_1,\xi_2,\ldots, \xi_{m+2}\}$. We shall assume not, and derive a contradiction. Suppose there are exactly two distinct points $p,q$ then by the preceding arguments $m$ is necessarily even, and $\xi_{2i-1} = p$ and $\xi_{2i} = q$ for each $1\leq i \leq \frac{m}{2}$. In that case, we can first show that the image of $h$ is \textit{exactly} the geodesic $\gamma$ between $p$ and $q$: arguing exactly as in \cite[Lemma 3.5]{HTTW95}, for any $z\in \C$  we can choose an exhaustion of $\C$ by polygons $\{G_k\} _{k\geq 1}$ each containing
$z$ and with a boundary comprising horizontal and vertical line segments, such that the distance $d(h(\partial G_k), \gamma) \to 0$ as $k\to \infty$. Since the distance function of $h(x)$ to $\gamma$ 
(as $x$ varies in $\C$) is subharmonic, we conclude that the distance of $h(z)$ to $\gamma$ must be zero, i.e. $h(z) \in \gamma$. Moreover, since $\gamma$ lies in a totally geodesic hyperbolic plane, we can apply Proposition \ref{est} to conclude that in each horizontal half-plane, the map $h$ approximates the collapsing map $(x,y) \mapsto x$ followed by an isometric embedding 
to $\gamma$. A calculation shows that the Hopf differential of this limiting map is the constant quadratic differential $\frac{1}{4} dz^2$, and therefore the Hopf differential of $q$ is bounded on each half-plane. The only such polynomial quadratic differential is the constant differential $cdz^2$ (for some constant $c$), which contradicts the assumption that $q$ is a degree-$m$ polynomial quadratic differential where $m\geq 1$. 

In the case that the polynomial Hopf differential has degree zero, i.e.\ the polynomial is a constant quadratic differential (with exactly two horizontal half-planes around $\infty$), the above argument shows that the image of $h$ is a geodesic line. 
\end{proof}

\noindent  \textit{Remark.} As mentioned in the proof above, the identity \eqref{hopf} and the exponential decay \eqref{eest} continue to hold in this case when the target is $\HH^3$ (see \cite[Equation (3.1 and Theorem 3.4)]{Min92}). We also have the analogue of \cite[Lemma 3.1]{HTTW95}, namely Lemma \ref{canoe}. Thus, we can establish the same statement as Proposition \ref{est}, but with $\HH^2$ replaced by $\HH^3$: the same proof carries through. In other words, in each horizontal half-plane of $q$, far from the zeros of $q$, the map $h$ is exponentially close to a collapsing map to the corresponding geodesic side of the twisted ideal polygon  $P$.

\section{Proof of Theorem \ref{main}}\label{sref}


As mentioned in \S1, the strategy of the proof is to construct a suitable initial $C^2$-smooth map $u_0:\C\to \HH^3$ asymptotic to the desired twisted ideal polygon (we do this in \S3.1), and then run the harmonic map heat flow \eqref{heat1}. The properties of the initial map together with Theorem \ref{wang} guarantee the long-time existence of the flow (\S3.2). In \S3.3, we show that each $u_t$ along the flow has the same tension-field decay and asymptotics as $u_0$. In \S3.4, we apply the maximum principle to show that the image of each $u_t$ is trapped in the convex hull of the vertices of the ideal polygon; this relies on the negative curvature of the target $\HH^3$ . In \S3.5, we first improve this by showing that in fact, the flow remains a uniformly bounded distance from the initial map, using a comparison map $P:\C \to \HH^3$ whose image is a \textit{pleated plane} asymptotic to the given twisted ideal polygon. Finally, we show that convergence indeed follows from these uniform estimates, and the limiting map has the desired asymptotics. 



\subsection{Construction of the initial map}\label{sstep2}
Let $P$ be the given  twisted ideal polygon in $\mathbb{H}^3$ with $n$ ideal vertices $\{\xi_1,\xi_2,\ldots,\xi_n\}$ in $\mathbb{CP}^1=\partial_{\infty}\mathbb{H}^3$ and $n$ geodesic sides $\{ \gamma_1, \gamma_2,\ldots, \gamma_n\}$ where $\gamma_i$ is a bi-infinite geodesic from $\xi_i$ to $\xi_{i+1}$ for each $i\in \{1,2,\ldots, n\}$, where the index set is cyclically ordered.

\subsubsection{Defining a planar polygon}
We first note that the twisted ideal polygon $P$ is obtained by bending an ideal polygon that is \textit{planar}, i.e.\ lies in a totally geodesic copy of the hyperbolic plane, which we denote by $H$, along (a subset of its) diagonals.

Here, a \textit{diagonal} of $P_0$ is a bi-infinite geodesic between two of its ideal vertices; it necessarily lies in the totally geodesic plane $H$. Also, a bending of $P_0$ along a diagonal $d$ is obtained by rotating the geodesic sides lying on one side of $d$ relative to those on the other side, where the rotation is an elliptic isometry of $\HH^3$ with axis $d$ (\textit{c.f.} Figure \ref{fig2}).

\begin{lem}[Theorem 5.1 of \cite{SGMahangrafting}]\label{bend} 
    There is a planar ideal polygon $P_0$ contained in a totally geodesic hyperbolic plane $H$ in $\HH^3$, such that $P$ is obtained by bending $P_0$ along a collection $\mathcal{C}$ of pairwise-disjoint diagonals of $P_0$.
\end{lem}

\begin{proof}[Idea of the proof]
Since the statement is immediately implied by that of \cite[Theorem 5.1]{SGMahangrafting},  we sketch the idea of the proof, and refer to that paper for details. In fact, for any choice of a  maximal set of  pairwise-disjoint diagonals $\mathcal{C}$ in an abstract $n$-sided ideal polygon $\mathcal{P}$ , one can determine such a planar ideal polygon $P_0$. Such a collection $\mathcal{C}$ necessarily has $(n-3)$ elements, and determines an ideal triangulation of $\mathcal{P}$.  The given twisted ideal polygon $P$ can be thought of as a map from the abstract ideal polygon $\mathcal{P}$ to $\HH^3$. 

Each  diagonal belongs to two adjacent ideal triangles of $\mathcal{P}$, and the corresponding vertices of $P$ determine four points in $\cp =  \partial_\infty \HH^3$. Taking the complex cross ratio of these four points, we obtain $n-3$ complex numbers. Indeed, one can reverse this process, and uniquely determine a twisted ideal polygon in $\HH^3$ (upto postcomposition by $\pslc$) from a $(n-3)$-tuple of complex numbers.   

The planar ideal polygon $P_0$ is obtained when the parameters are the modulus of these complex numbers. This will lie on a totally geodesic  hyperbolic plane since these parameters (which are cross-ratios of two adjacent ideal triangles) are all real and positive. Geometrically, a complex cross-ratio $c = r\exp{i\theta}$  encodes the ``shear-bend" parameters between the two adjacent ideal triangles, and $\theta$ is the angle between the geodesic planes where they lie (see the discussion regarding grafting ideal quadrilaterals just before the proof of Theorem 5.1 in \cite{SGMahangrafting}). 
\end{proof}

\begin{figure}[t]
 \labellist
 \small\hair 2pt
\pinlabel $P_0$ at 40 130
\pinlabel $P$ at 440 135
 \endlabellist 
\centering
\includegraphics[scale=.6]{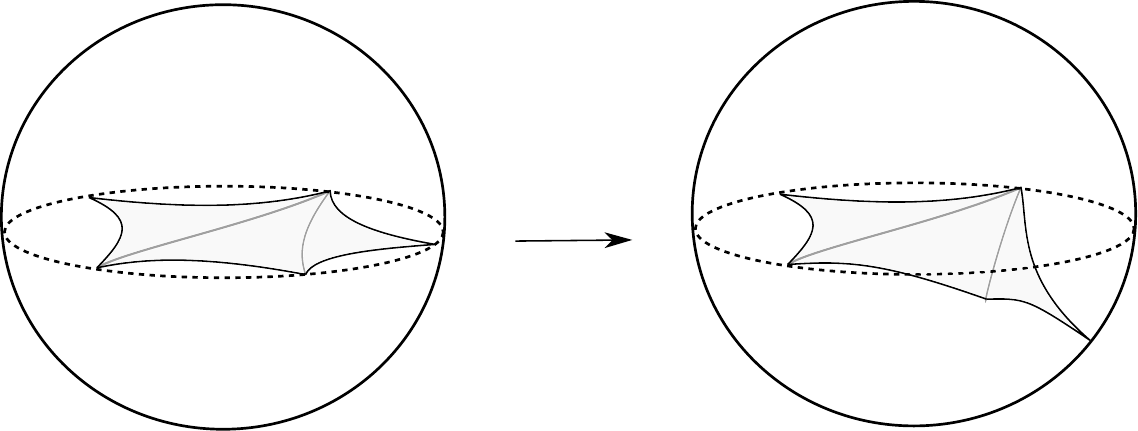}
\caption{The twisted ideal polygon $P$ is obtained by bending a planar ideal polygon $P_0$ along a collection of diagonals (see Lemma \ref{bend}). The shaded surface on the right obtained from the planar region bounded by $P_0$ after bending, is a pleated plane (\textit{c.f.} Definition \ref{mapP}).} 
\label{fig2}
\end{figure}

We shall now choose a harmonic map $h:\C \to H $ which is asymptotic to the ideal polygon $P_0$, such that the Hopf differential of $h$ is a polynomial quadratic differential. From the discussion in \S2.2, this polynomial differential is necessarily of the form $q(z)dz^2$ where $q(z)$ is a polynomial of degree  $(n-2)$. (Recall that $P_0$ is an ideal polygon with $n$ sides.) Such a harmonic  map exists by Theorem \ref{httw}, and it is easy to see (by comparing the dimensions of the space of such polynomial differentials on one hand, and the space of ideal $n$-gons on the other) that it is not unique. In fact, by \cite[Proposition 3.12]{Gupta21}, there is a unique  $h$ such that the Hopf differential has a prescribed {principal part} (as defined in Definition \ref{defn:princ}) -- we shall use this flexibility in Proposition \ref{prop:princ}. 


\smallskip 

In what follows, we shall modify this map $h$ to obtain our initial map $u_0$.

\subsubsection{Decomposing the domain}

Let $q= q(z)dz^2$ be the Hopf differential of the above harmonic map $h:\mathbb{C}\rightarrow \HH^2$, where $q(z)$ is a polynomial of degree $(n+2)$.  Recall from \S2.2 that in the induced singular-flat geometry, there are $n$ horizontal and $n$ vertical half-planes arranged in a cyclic order around infinity. 

Choose horizontal leaves $L_1, L_2,\cdots, L_n$ in each of the horizontal half-planes, in cyclic order, at a distance $R\gg 0$ from the set of zeros of $q$.  Denote by $H_i$ the horizontal half-plane bounded by the leaf $L_i$.

Similarly, in each vertical  half-plane around $\infty$, choose a bi-infinite vertical line $V_i$  that intersects $L_i$ and $L_{i+1}$, at a distance at least $R$ from the zeros of $q$, and denote the (smaller) vertical half-plane that it bounds by $C_i$.

We thus obtain a cyclically ordered chain  $\{C_1, H_1, C_2, H_2,\ldots C_n, H_n\}$ of overlapping half-planes such that each intersects the next along a quarter-plane. Note that the union of these half-planes is $C\setminus K$, where $K$ is a compact set. 

\subsubsection{Defining the map}
We shall define the initial map $u_0:\C \to \HH^3$ by first defining it on each of the half-planes $C_i$ and $H_i$ in the above decomposition, and then extending it the whole of $\C$. 

We start with the harmonic map $h:\C \to \HH^2$ defined in the previous subsection.
From the asymptotic behaviour of the harmonic map $h$ (discussed in \S2.2), we know that each vertical half-plane  $C_i$ maps into a cusp of the ideal polygon $P_0$,  i.e.\  a region of $H$ bounded by two geodesic sides that are asymptotic to the $i$-th ideal vertex, and an arc of a horocycle centered at that vertex. We shall denote such a cusp of $P_0$  by  $\widetilde{C}_i$.

 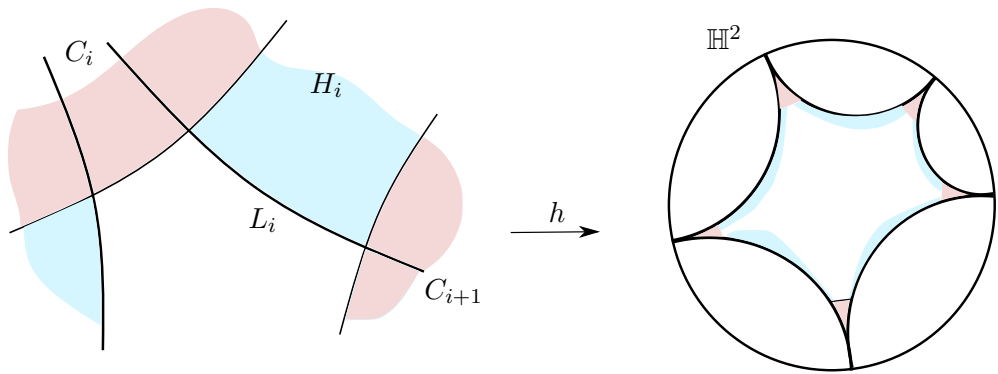
\begin{figure}[t]
 \labellist
 \small\hair 2pt
\pinlabel $H_i$ at 200 180
\pinlabel $C_i$ at 45 200
\pinlabel $L_i$ at 160 95
\pinlabel $\mathbb{H}^2$ at 450 210
\pinlabel $C_{i+1}$ at 280 50

\pinlabel $h$ at 345 100

 \endlabellist 
\centering
\includegraphics[scale=.6]{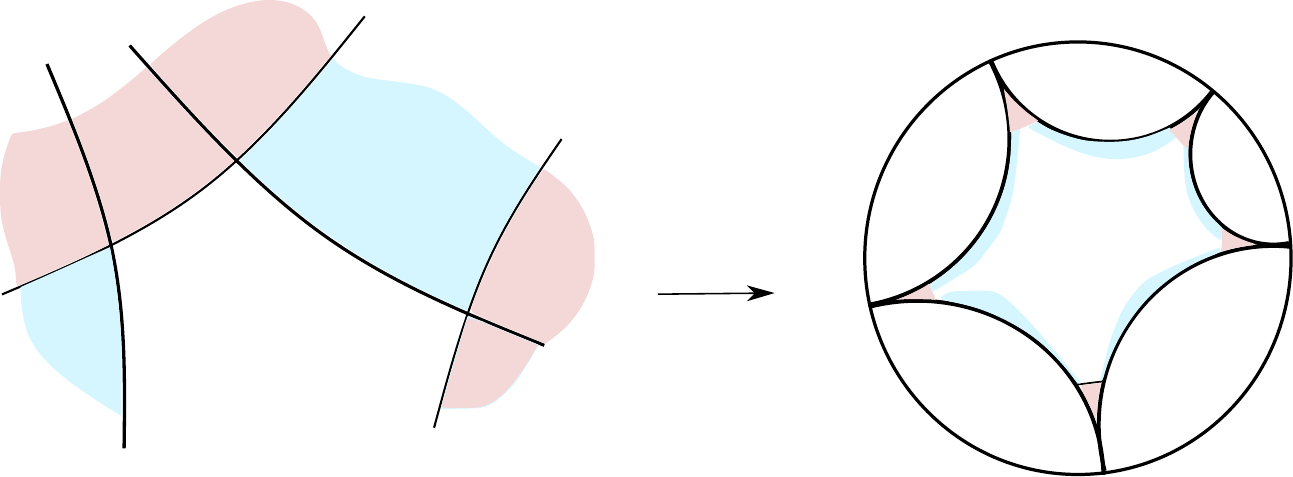}
\caption{The harmonic map $h:\C \to \HH^2$ takes each vertical half-plane $C_i$ in the domain  (shaded red) to a cusp of $P_0$, and each horizontal half-plane $H_i$ to a neighborhood of the $i$-th geodesic side.} 
\label{fig3}
\end{figure}

At each ideal vertex $\xi_i$ of the twisted ideal polygon $P$ in $\HH^3$,  we shall also consider a planar cusped region $\widehat{C}_i$ that is defined as follows: Assume that $\xi_i$ is at infinity in the upper half-space model of $\HH^3$; the geodesic sides $\gamma_{i}$ and $\gamma_{i+1}$ are then vertical lines contained in a totally-geodesic hyperbolic plane $V$. The cusped region $\widehat{C}_i$ is defined to be the subset  of $V$ bounded by the two geodesics and a horocylic line at a height chosen such that the cusps $\widetilde{C}_i$ and $\widehat{C}_i$ are isometric. 

We observe that since $P$ is obtained by bending $P_0$ along diagonals (see Lemma \ref{bend}), the lengths of the geodesic segments along the $i$-th side of $P_0$ and the $i$-th side of $P$, that are  disjoint from the cusps defined above, are exactly the same. 

In what follows, to write down the map,  we shall identify the horizontal half-plane $H_i = \{(x,y)\in\mathbb{R}^2:x\in\mathbb{R}, y>0\}$ and identify the cusp $\widehat{C}_i$ with a subset $\{(x,0,t)\in \mathbb{H}^3:|x|< 1, t>t_0>0\}$.  Let the restriction of the  harmonic map $h$ to $H_i$ be 
\begin{equation}\label{h1}
h_i(x,y)=(f^i(x,y), 0, g^i(x,y)),
\end{equation}
in the upper half-space model of $\HH^3$,  where we assume that the totally geodesic copy of $\mathbb{H}^2$ containing $P_0$ is the vertical plane $y=0$.  By Proposition \ref{est}, we also know that $f^i(x,y) \to 0$ and $g^i(x,y) \to e^{x}$ as $(x,y)$ diverges, exponentially fast in terms of the distance $\lVert (x,y) \rVert$ - we shall use that in the next section.

We  define the initial map $u_0:\mathbb{C}\rightarrow\mathbb{H}^3$  on each $C_i$ to be the restriction of the harmonic map $h$ on $C_i$  post-composed with an isometry that maps $\widetilde{C}_i$ to $\widehat{C_i}$. Note that since we are post-composing with an isometry, each such map is again harmonic.

We will now define the map $u_0$ on each horizontal half-plane. 
On $C_i\cap H_i$ and $C_{i+1} \cap H_{i}$, $u_0$ is already defined; we can assume that these are the quarter-planes $\{(x,y)\in \mathbb{R}^2\ \vert\ x>a_i, y>0\}$ and $\{(x,y)\in \mathbb{R}^2\ \vert\ x< a_{i+1} , y>0\}$ respectively, both contained in $H_i$. What remains is the half-infinite strip $[a_i,a_{i+1}]\times[0,\infty)$ between $C_i\cap H_i$ and $C_{i+1}\cap H_{i}$ that we shall denote by $D_i$.
(See the left side of Figure \ref{fig4}.)

Let $\gamma_i$ be the geodesic line common to the two successive cusps $\widehat{C}_i$ and $\widehat{C}_{i+1}$. There is an angle $\theta_0>0$ such that the elliptic rotation $R_{\theta_0}$ about the axis $\gamma_i$ takes the plane containing $\widehat{C}_i$ to the plane containing $\widehat{C}_{i+1}$.  We can assume that $u_0$ is given by \eqref{h1} on $C_i \cap H_i$ and on $C_{i+1}\cap H_{i}$ is the map $h$ rotated by an angle $\theta_0$:
\begin{equation}\label{h2}
h_i^{\theta_0}(x,y)=(f^i(x,y)\cos\theta_0,f^i(x,y)\sin\theta_0,g^i(x,y))
\end{equation}

To define $u_0$ from the remaining half-strip $D_i$ to $\mathbb{H}^3$, that we shall denote by $u_i$,  we shall interpolate between the maps $h_i$ and $h_i^{\theta_0}$, by rotating the original map $h_i$ by an angle that varies from $0$ on the left quarter-plane, to $\theta_0$ on the right quarter-plane. 

 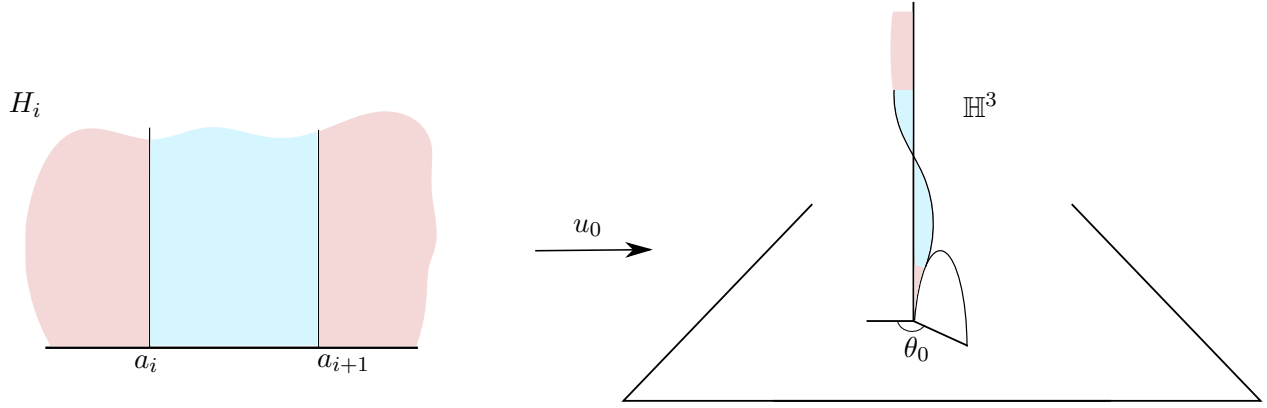
\begin{figure}[t]
 \labellist
 \small\hair 2pt
\pinlabel $H_i$ at 0 140
\pinlabel $\mathbb{H}^3$ at 450 140
\pinlabel $\theta_0$ at 420 25
\pinlabel $a_i$ at 57 18
\pinlabel $a_{i+1}$ at 150 18
\pinlabel $u_0$ at 265 82

 \endlabellist 
\centering
\includegraphics[scale=0.8]{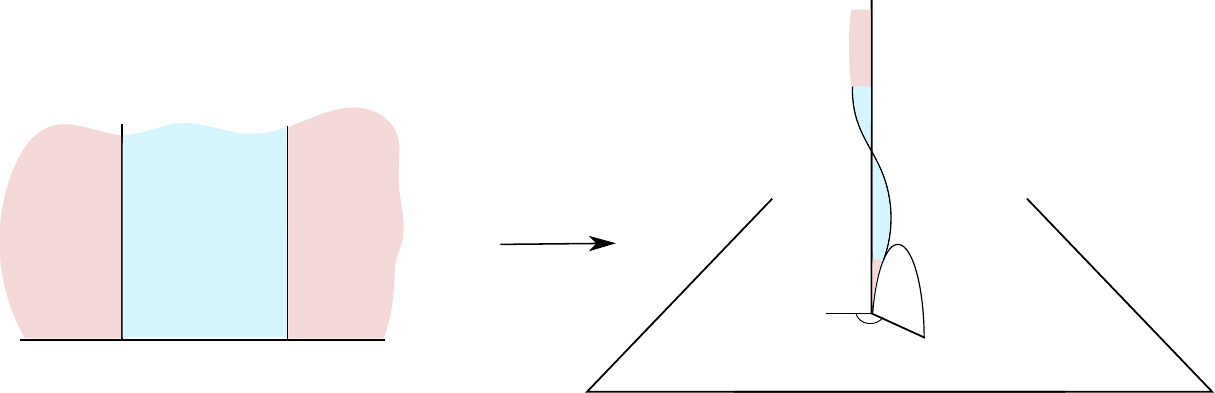}
\caption{On each horizontal half-plane $H_i$, the initial map $u_0$ is obtained by modifying $h$ such that it interpolates between the maps from the vertical half-planes $C_i$ and $C_{i+1}$ to the cusps of $P$. Here, the $i$-th cusp and $(i+1)$-th cusp of $P$ shown on the right lie on totally geodesic planes intersecting at an angle $\theta_0$.}
\label{fig4}
\end{figure}

To describe this, let $(a,b)$ be an open interval containing $[a_i,a_{i+1}]$, and 
choose a $C^2$-smooth function $\theta^i:(a,b)\rightarrow [0,\theta_0]$ such that
\begin{enumerate}[(i)]
\item $\quad \theta^i(x)=0 \quad\text{for all} \quad x\leq a_{i+1} $,\quad $\text{and}\quad\theta^i(x)=\theta_0 \quad\text{for all}\quad x\geq a_{i}$,
\item  The derivatives $\quad (\theta^i)'\quad \text{and} \quad (\theta^i)''\quad \text{ are bounded with}\quad (\theta^i)''(a_i)=0=(\theta^i)''(a_{i+1})$.
\end{enumerate} 

The map $u_i$ is then defined by
\begin{equation}\label{h3}
u_i=(u^1_i,u^2_i,u^3_i)=(f^i(x,y)\cos\theta^i(x),f^i(x,y)\sin\theta^i(x),g^i(x,y)).
\end{equation} 
In view of  $(i)$ and $(ii)$, $h_i, h_i^{\theta_0}$ and $u_i$ together define a $C^2$-smooth map on $H_i$. 

\vspace{.05in} 

Completing the above construction for each $1\leq i\leq n$,  we obtain a $C^2$-smooth map defined on the chain of half-planes $\{C_1, H_1, C_2, H_2,\ldots C_n, H_n\}$ in $\mathbb{C}$. Finally, we choose a $C^2$-smooth extension to the compact complement $K$ to obtain an initial map $u_0:\mathbb{C}\rightarrow \mathbb{H}^3$.

\subsection{Existence of the flow}

In the previous section we defined a $C^2$-smooth map $u_0:\C \to \HH^3$ that is asymptotic to the given twisted ideal polygon $P$.
In this subsection we shall prove that the harmonic map heat flow \eqref{heat1} with initial map $u_0$ exists for all time, by an application of Theorem \ref{wang}. For this, we first establish some properties of $u_0$, including the exponential decay of its tension field. 

\vspace{.05in}

Throughout, let  $\sigma\lvert dz\rvert ^2$ be the conformal metric on the complex plane $\C$ which is obtained by a $C^2$-smoothening  of the singular-flat metric induced by the Hopf differential of the harmonic map $h:\C \to \HH^3$ mapping into a totally geodesic plane, that we introduced in \S3.1.1. 
The smoothening is done locally, in a neighborhood of each singularity that is contained in the interior of the compact set $K$, away from the chain of half-planes $\{C_1,H_1,C_2,H_2,\ldots, C_n,H_n\}$.

Note that with this metric $(\C, \sigma)$ is complete, since  the original Hopf differential metric is complete, and there are finitely many singularities. Moreover, since the singularities have a cone-angle that is greater than $2\pi$, the smoothened metric $\sigma$ is negatively curved in neighborhoods of those points, and flat in their complement.

We first start with the boundedness of the energy density of the initial map $u_0$.

\begin{lem}\label{intped}
The energy density $e(u_0)$ is uniformly bounded on $\C$. 
\end{lem}
\begin{proof}
Consider the restriction of $u_0$ to a horizontal half-plane $H_i$; recall that this map is obtained by starting with $h$ and modifying it. Indeed, it was defined in the three pieces $C_i \cap H_i, C_{i+1} \cap H_i$ and $D_i$ constituting $H_i$ by the equations \eqref{h1},\eqref{h2} and \eqref{h3}. Recall also that the original map $h_i = h\vert_{H_i}$ is harmonic, and is given by the equation \eqref{h1}. 

Since the conformal metric $\sigma$ on $\C$ coincides with the metric induced by the Hopf differential of $h$, and $H_i$ is a horizontal half-plane in that metric, the energy density $e(h_i)$ is estimated by \eqref{ede}, and in particular, is uniformly bounded. 

This immediately implies that the energy density of $u_0$ is uniformly bounded on $C_i \cap H_i$ and  $C_{i+1} \cap H_i$, since it coincides with $h_i$ in the former subset, and with $h_i$ post-composed by an isometry (namely, an elliptic rotation) in the latter region.  Here we are using the fact that post-composition with an isometry does not affect the energy density. 

It remains to derive an estimate for the energy density in the half-infinite strip $D_i$, where the map interpolates between the two maps just mentioned. Denote the restriction $u_0\vert_{D_i}$ by $u_i$. For ease of notation, in the following calculation, we shall  drop the subscript in what follows, and set $u_i=u = (u^1,u^2, u^3)$ and $f^i=f,g^i=g,\theta^i=\theta$. Computing the partial derivatives, we obtain: 
\[
 \frac{\partial u^1}{\partial x}=f_x\cos\theta(x)-f(x,y)\sin\theta(x)\theta'(x),\frac{\partial u^1}{\partial y}=f_y\cos\theta(x),\]
\[\frac{\partial u^2}{\partial x}=f_x\sin\theta(x)+f(x,y)\cos\theta(x)\theta'(x),\frac{\partial u^2}{\partial y}=f_y\sin\theta(x).\]  From this, we calculate $$e(u_0)=\frac{1}{g^2}\sum_{k=1}^{2}\left[\left(\frac{\partial u^k}{\partial x}\right)^2 +\left(\frac{\partial u^k}{\partial y}\right)^2\right]=e(h_0)+\frac{f^2\theta'^2}{g^2}.$$ Now since $h_i=(f_i,g_i)$ is  asymptotically $C^1$-close to the collapsing map by Proposition  \ref{est}, we know that  $f \sim 0$ and $g \sim e^{x}$ at any point $(x,y)$ of large norm. Moreover, we have chosen the function $\theta$ to have a uniformly bounded first derivative. From these, it follows that there exists $E>0$ such that $e(u_0)\leq E$.
\end{proof}

Next, we show that the tension field of the initial map $u_0$ has an exponential decay in $\C$. Since on each vertical half-plane $C_i$ the map $u_0$ is the harmonic map $h$ post-composed by an isometric embedding to the cusp $\widehat{C}_i$, the tension field there is zero, and it remains to show:

\begin{lem}\label{intptf}
The norm of the tension field $|\tau(u_0)|$ decays exponentially in each half-plane $H_i$.
\end{lem}
\begin{proof}
Recall that $u_0|_{H_i}$ consists of the maps $h_i,h_i^{\theta_0},u_i$ of the form \eqref{h1},\eqref{h2},\eqref{h3}. As $h_i$ and $h_i^{\theta_0}$ are harmonic, their tension field vanishes. So we need to compute the tension field of $u_i$ on $D_i$. Once again, for ease of notation, we drop the subscript $i$, and write $u$ instead of $u_i$ et cetera. In what follows $\Delta_0$ is the usual Laplacian on $\C$ and $\Delta$ is the Laplacian for the conformal metric $\sigma$. The components of the tension field of $u$ are given by
\begin{align*}
&\tau(u)^1 =\Delta u^1+\sigma^{ij}(x)\frac{\partial u^{\beta}}{\partial x^i}\frac{\partial u^{\gamma}}{\partial x^j}\Gamma'^1_{\beta \gamma}(u(x))\\
&=\sigma^{11}\Delta_0(f(x,y)\cos\theta(x))+\sigma^{ii}(x)\frac{\partial u^1}{\partial x^i}\frac{\partial u^3}{\partial x^i}\Gamma'^1_{13}(u(x))\quad(\text{since}\quad \sigma^{ij}=0\quad \text{for} \quad i\neq j,\sigma^{11}=\sigma^{22})\\
&=\sigma^{11}\cos\theta(x)\Delta_0 f-2f_x\sin\theta(x)\theta'(x)-f(x,y)(\cos\theta(x)\theta'^2(x)+\theta''(x)\sin\theta(x))\\
&+2\cos\theta(x)(\Gamma'^1_{13}f_xg_x+\Gamma'^1_{13} f_yg_y)
-2\Gamma'^1_{13}f(x,y)\sin\theta(x)\theta'(x)\\
&=\cos\theta(x)\tau(h_0)^1-f(x,y)(\cos\theta(x)\theta'^2(x)+\theta''(x)\sin\theta(x))+\frac{2}{g}f(x,y)\sin\theta(x)\theta'(x)-2f_x\sin\theta(x)\theta'(x)\\
&= -f(x,y)(\cos\theta(x)\theta'^2(x)+\theta''(x)\sin\theta(x))+\frac{2}{g}f(x,y)\sin\theta(x)\theta'(x)-2f_x\sin\theta(x)\theta'(x).
\end{align*}
Similarly, $\tau(u)^2=f(x,y)(-\sin\theta(x)\theta'^2(x)+\theta''(x)\cos\theta(x))-\frac{2}{g}f(x,y)\cos\theta(x)\theta'(x)+2f_xcos\theta(x)\theta'(x)$.
Finally, 
\begin{align*}
&\tau(u)^3 =\Delta u^2+\sigma^{ij}(x)\frac{\partial u^{\beta}}{\partial x^i}\frac{\partial u^{\gamma}}{\partial x^j}\Gamma'^3_{\beta \gamma}(u(x))\\
&=\Delta_0g+\sigma^{ii}(x)\frac{\partial u^1}{\partial x^i}\frac{\partial u^1}{\partial x^i}\Gamma'^3_{11}(u(x))+\sigma^{ii}(x)\frac{\partial u^2}{\partial x^i}\frac{\partial u^2}{\partial x^i}\Gamma'^3_{22}(u(x))+\sigma^{ii}(x)\frac{\partial u^3}{\partial x^i}\frac{\partial u^3}{\partial x^i}\Gamma'^3_{33}(u(x))\\
&=\Delta_0g+\frac{1}{g}(|\nabla_0u^1|^2+|\nabla_0u^2|^2-|\nabla_0u^3|^2)\\
&=\Delta_0g+\frac{1}{g}(|\nabla_0f|^2-|\nabla_0g|^2)+\frac{f^2\theta'^2}{g} =\tau(h_0)^3+\frac{f^2\theta'^2}{g}= \frac{f^2\theta'^2}{g}.
\end{align*}
Observe that, $x$ stays in the bounded interval where the interpolation is happening. 
Now, 
\begin{align}\label{dlemma}
\nonumber|\tau(u)|^2&=\frac{1}{g^2}\left(|(\tau(u)^1)|^2+|(\tau(u)^2)|^2+|(\tau(u)^3)|^2\right)\\
&\leq \frac{C}{g^2}\left(f^2\theta'^4+f^2\theta''^2+f_x^2\theta'^2+\frac{f^2\theta'^2}{g^2}+\frac{f^4\theta'^4}{g^2}\right).
\end{align} 
for some constant $C>0$. 

Using the fact that the derivatives of $\theta$ are bounded, and (i) and (ii) of Proposition \ref{est}  imply that $f$ and $f_x$ have an exponential decay, we obtain 
\begin{equation}\label{uselet2}
|\tau(u_0)|^2(z)\leq C_0e^{-\delta R(z)}
\end{equation}
for some $\delta>0, C_0 >0$.
\end{proof}
\begin{cor}
 The norm of the tension field $|\tau(u_0)|\in L^p(\mathbb{C})\cap L^{\infty}(\mathbb{C})$, for $p\geq 1$.
\end{cor}

The preceding lemmas now imply:

\begin{prop}
    The harmonic map heat flow \eqref{heat1} starting with the initial map $u_0:\C \to \HH^3$ constructed in \S3.1 has a long-time solution $u: \C \times [0,\infty) \to \HH^3$ which is also unique.
\end{prop}
\begin{proof} To apply Theorem \ref{wang} we check that its hypotheses hold. Indeed,
\begin{equation}\label{bbound}
b^2(x,t)=\int_\mathbb{C}H(x,y,t)|\tau(u_0)|^2(y)dy \leq \lVert \tau(u_0) \rVert_\infty^2\int_\mathbb{C}H(x,y,t)dy= C^2 < \infty 
\end{equation}
where the supremum norm of the tension field $C = \rVert \tau(u_0) \rVert_\infty $ is finite by the previous lemma.  The long-time solution thus exists, and satisfies \begin{equation}\label{ico11}
|\tau(u)(x,t)|\leq b(x,t).
\end{equation} 
We shall derive a better estimate for this later. 

Note that for any $T>0$ \[\int_{0}^{T}\int_\mathbb{C} e^{-r^2(x)}b^2(x,t)dxdt\leq C^2 \int_{0}^{T}\int_\mathbb{C} e^{-r^2(x)}dxdt<\infty\]  and $\mathbb{H}^3$ is simply connected. Thus,  the uniqueness statement in  Theorem \ref{wang} also holds. 
\end{proof}

\subsection{Estimates along the flow}\label{esubs}

In the subsection, let $u_t: (\C, \sigma) \to \HH^3$ be the time-$t$ map of the harmonic map heat flow with initial map $u_0$. We shall establish properties of $u_t$, including that it has a bounded energy density (independent of $t$), an exponentially decaying tension field, and is asymptotic to the initial map $u_0$. 

\vspace{.05in}

For the first lemma on the energy density, we shall use the following result, known as Moser's Harnack inequality for subsolutions of the heat equation on a Riemannian manifold $M$.

\begin{prop}{[\cite{WH}, Proposition 7.5, p.268}]\label{mh}
Let $v\in C^{\infty}(M\times[t_0-R^2,t_0])$ be a non-negative function  satisfying\[\left(\Delta-\frac{\partial}{\partial t}\right)v\geq-C v,\] for some $C>0$ on $ B_R(x_0)\times[t_0-R^2,t_0]$. Then there exists a positive constant $C_2$ such that\[v(x_0,t_0)\leq C_2R^{-(m+2)}\int_{t_0-R^2}^{t_0}\int_{B_R(x_0)}v(y,s)dvol_g(y) ds,\] where $m$ is the dimension of $M$.
\end{prop}

\begin{lem}\label{edb}
The energy density $e(u_t)$ of the solution $u_t$ is uniformly bounded on $\mathbb{C}\times [0,\infty)$.
\end{lem}
\begin{proof}
By Weitzenbock Formula for $e(u_t)$ \cite[Proposition $4.2$]{Nishikawa02}, we have that the energy density satisfies the equation 
\begin{align}\label{wbock1}
\left(\Delta-\frac{\partial}{\partial t}\right)(e(u_t))&=|\nabla \nabla u_t|^2+\sum_{i=1}^{2}\left\langle du_t\left(\sum_{j=1}^{2}Ric^{\sigma}(e_i,e_j)e_j\right),du_t(e_i) \right\rangle\\\nonumber
&-\sum_{i,j=1}^{2}\langle R^{\mathbb{H}^3}(du_t(e_i),du_t(e_j))du_t(e_j),du_t(e_i)\rangle.
\end{align}
where $e_1,e_2$ is an orthonormal basis in the tangent space of the point in the domain. 

Since the first term in the R.H.S of \eqref{wbock1} is non-negative, the Ricci curvatures of the domain metric $\sigma$ are bounded below (by some negative constant)  and $\mathbb{H}^3$ is negatively curved,  we conclude that 
\begin{equation}\label{sr}
\left(\Delta-\frac{\partial}{\partial t}\right)(e(u_t))\geq -ke(u_t)
\end{equation} 
for some constant $k>0$.

Recall that the metric $\sigma$ in the domain is obtained by a smoothening of the singular-flat metric induced by the Hopf differential of the harmonic map $h$. 
Let  $M_{\epsilon}=\mathbb{C}\setminus N(Z, \epsilon)$, where $N(z,\epsilon)$ is an $\epsilon$-neighborhood of the set of zeros of the Hopf differential. Here, we choose $\epsilon$ such that the metric $\sigma$ coincides with the Hopf differential metric outside $ N(Z, \epsilon)$. 
Applying Proposition \ref{mh} to $(x_0,t_0)\in M_{\epsilon}\times (\delta,T)$, for $R<\sqrt{\delta}$, we obtain
\begin{equation}\label{peh}
e(u_{t_0})(x_0)\leq  C_2R^{-4}\int_{t_0-R^2}^{t_0}\int_{B_{R}(x_0)}e(u_s)dvol_g ds \leq 
C_2R^{-2}\int_{B_{R}(x_0)}e(u_0)dvol_g = C_2 R^{-2}E_{R}(u_0)
\end{equation} 
where $E_R(f)$ denotes the energy of the restriction of $f$  to the ball $B_R(x_0)$. For the second inequality above, we have used the fact that $E_R(u_t) \leq E_R(u_0) $; indeed,  the first variation of energy shows that this total energy is decreasing in $t$: \[\frac{d}{dt}E_R(u_t)=-\int_{B_R(x_0)}\left\langle\frac{\partial u_t}{\partial t},\tau(u_t)\right\rangle dvol_g=-\int_{B_R(x_0)}|\tau(u_t)|^2dvol_g\leq 0.\] 

From Lemma \ref{intped}, there exists $E>0$ such that $e(u_0)(x)\leq E$ for all $x\in \mathbb{C}$.  Since the metric $\sigma$ on the domain is obtained by a smoothening of a singular-flat metric,  there exists a constant $Q$ independent of $x_0$ such that
$\text{Vol}(B_R(x_0))\leq QR^2$. Now putting all these pieces of information in \eqref{peh},
$e(u_{t_0})(x_0)\leq  C_2E Q$.
Taking $\delta\rightarrow 0,T\rightarrow \infty$ we obtain a bound for the energy density of $u_t$ on $M_{\epsilon}\times [0,\infty)$.

By compactness, $e(u_t)$ is bounded on $N(Z,\epsilon)\times [0,1]$ by a positive constant, that we denote by $M_1$. Since $\mathbb{C}= N(Z, \epsilon) \cup M_\epsilon $, we have proved that $e(u_t)$ is uniformly bounded on $\mathbb{C}\times [0,1]$.
Applying Proposition \ref{mh} again to $(x_0,t_0)\in \mathbb{C}\times(1,T)$, for $R<1$, we obtain \[e(u_{t_0})(x_0)\leq CR^{-2}E_R(u_0)\le M_2,\] for some $M_2>0$.  Take $M_3=\text{max}\{C_2EQ,M_1,M_2\}$. Then $e(u_t)(x)\leq M_3$ for all $(x,t)\in \mathbb{C}\times [0,\infty)$ where $M_3$ is a constant independent of $x$ and $t$.
\end{proof}


We also observe:

\begin{lem}\label{wuie}
The $L^{\infty}$ norm of the tension field of $u_t$ satisfies $||\tau(u_t)||_{\infty}\leq St^{-\frac{1}{2}}$ for $t>0$.
\end{lem}
\begin{proof}
 Recall that the metric $\sigma$ on the domain $\C$ is a smoothening of the flat metric induced by a polynomial quadratic differential. Such s metric has bounded geometry, i.e\  its curvature is bounded uniformly from below and its injectivity radius is bounded uniformly away from $0$ on all of $\C$.
 Then  for all $(x,y,t) \in \C \times \C \times (0,\infty)$ the following heat kernel estimate holds (see, for example, \cite[Theorem 3] {CF91}): 
\begin{equation}\label{ari}
H(x,y,t)\leq c t^{-1/2}.
\end{equation}
for some constant $c>0$.

 Moreover, by Theorem \ref{wang}, we know that ,\[\left|\frac{\partial u}{\partial t}(x,t) \right|^2 = \left|\tau(u_t)(x) \right|^2 \leq b^2(x,t) = \int_{\mathbb{C}}H(x,y,t)|\tau(u_0)|^2(y)dy.\]  
Using \eqref{ari} and the fact that $|\tau(u_0)|\in L^2(\mathbb{C})$ we conclude that there is a constant $A>0$ such that $|\tau(u_t)|^2(x)\leq At^{-1/2}$, for all $x\in \C$. 
\end{proof}
Our next goal is to show that for any fixed $t>0$, the norm of the tension field  $|\tau(u_t)|$ is exponentially decaying in the space variable.
We shall use the following heat kernel estimate, which follows from \cite[Theorem 3.3]{LiYauparabolickernel} (see Equation (3.7) and the proof of Corollary 3.6 in \cite{JiapingWang}):

\begin{prop}\label{Kest}
Let $M$ be a complete Riemannian manifold with Ricci curvature bounded below by $-k$, where $k>0$. Then the heat kernel  $H(x,y,t)$  of $M$ satisfies
\begin{equation}\label{hk1}
H(x,y,t)\leq cV_x^{-\frac{1}{2}}(\sqrt{t})V_y^{-\frac{1}{2}}(\sqrt{t})\exp\left(-\frac{r^2(x,y)}{5t}+ctk\right),
\end{equation}
where $V_x(\sqrt{t})$ denotes the area of a ball centered at x with radius $\sqrt{t}$, the quantity $r(x,y)$ is the distance from $x$ to $y$ and $c$ is a positive real number. 
\end{prop}

\noindent \textit{Remark.} In our case we shall apply the above estimate to the manifold $(\C, \sigma)$. Note that the hypothesis hold, since the metric $\sigma$ is flat outside of a compact set. 

\vspace{.05in}

\begin{lem}\label{uboundtf}
For any $t>0$, the norm of the tension field $|\tau(u_t)|(x)$  decays exponentially in terms of the distance $r(0,x)$ in the conformal metric $\sigma$ on $\C$.
\end{lem}
\begin{proof}\label{tfb}
Note that the boundedness of the tension field follows from Lemma \ref{wuie}. Using Proposition \ref{Kest} and the remark following it, we can write:
\begin{equation}\label{hkbound}
H(x,y,t)\leq M_1e^{-M_2r^2(x,y)},
\end{equation}  
where $M_1,M_2$ are constants independent of $x$ and $y$. Recall that Theorem \ref{wang} asserts that 
\begin{equation}
    \lvert \tau(u_t)(x)\rvert^2 \leq \int_{\mathbb{C}}H(x,y,t)|\tau(u_0)|^2(y)dvol(y)
\end{equation}
and thus it remains to use \eqref{hkbound} and Lemma \ref{intptf} to estimate the right-hand integral. Indeed, 
we have,
\begin{align}
\nonumber\int_{\mathbb{C}}H(x,y,t)|\tau(u_0)|^2(y)dvol(y)
&\leq M_1\int_{K}e^{-M_2r^2(x,y)}|\tau(u_0)|^2(y)dvol(y) \\\nonumber
&+ M_1\int_{\mathbb{C}\setminus K}e^{-M_2r^2(x,y)}|\tau(u_0)|^2(y)dvol(y)\\
&=M_1 (I+II)
\end{align} 
where $K$ is the compact set introduced in \S3.1.2, the complement of which is a union of horizontal and vertical half-planes $\C \setminus K = \bigcup\limits_{i=1}^n C_i \cup H_i$.

\noindent  Using the inequality $r(x,y)\geq r(0,y)-r(0,x)$  we have,
\begin{align}\label{roman1}
\nonumber &I \leq M_1\int_{K}e^{-M_2(r^2(0,y)-2r(0,x)r(0,y)+r^2(0,x))}dvol(y)\\\nonumber
&\leq M_1e^{-M_2 r^2(0,x)}\int_{K}e^{-M_2(r^2(0,y)-2r(0,x)r(0,y))}dvol(y)\\
&\leq C(t) e^{-M_2 r^2(0,x)},
\end{align}
for some constant $C(t)>0$. 

Since the tension field on each vertical half-plane $C_i$ is zero, it remains to estimate $II$ on each horizontal half-plane $H_i$ and add, using the exponential decay of the tension field \eqref{uselet2} from Lemma \ref{intptf}:
\begin{equation}\label{roman2}
 II 
 \leq C_0\sum_{i=1}^{n}\int_{H_i}e^{-M_2(r(0,y)-r(0,x))^2} e^{-\delta r(0,x)}dy \leq D(t) e^{-\delta^\prime r(0,x)}
 \end{equation}
 for some constants $D(t), \delta^\prime >0$.  
 
 The desired exponential decay then follows from \eqref{roman1} and \eqref{roman2}.
\end{proof}

\begin{cor}\label{tend0}
For any fixed $t>0$, we have the distance estimate \[d(u_0(x),u_t(x)) = O(e^{-\alpha_1 r(0,x)}) \] for a constant $\alpha_1>0$. (As before, $r(0,x)$ is the distance in the conformal metric $\sigma$ on $\C$.) 
\end{cor}
\begin{proof}
We have 
\begin{equation}\label{distb}
d(u_0(x),u_t(x))\leq \int_{0}^t\left|\frac{\partial u_s}{\partial t}\right|(x)ds = \int_{0}^t|\tau(u_s)|(x)ds 
\end{equation} 
using the fact that $u_t$ satisfies \eqref{heat1}.
Then, we apply the fact that for any $s>0$, the norm of the tension field $\lvert \tau(u_s)\rvert$ decays exponentially in the space variable, by the preceding Lemma.
\end{proof}


\subsection{Image is trapped in a convex hull}\label{weuse}

 Recall that $\xi_1,\xi_2,\cdots,\xi_n$ are the ideal vertices of the twisted ideal polygon $P$ in $\HH^3$. Let $Q$ be the convex hull of these ideal points in $\partial_\infty \HH^3$. In this subsection we will prove that the image of $u_t$ is trapped in a fixed neighborhood of $Q$, for all $t\in \mathbb{R}^+$.

\vspace{.05in}

Let $Q = \bigcap_{\alpha \in \Lambda} H_\alpha$ where each $H_\alpha$ is a half-space in $\HH^3$ bounded by a totally geodesic plane, and $\Lambda$ is some index set.
 We start with a basic convexity property of the distance function from each such half-space in $\HH^3$ (or more generally, a convex set in a negatively-curved space):
 
\begin{lem}[and Definition of $f_\alpha$]\label{n1}
For each $\alpha \in \Lambda$, the distance function $f_{\alpha}(x):=d(x,H_{\alpha})$ is convex.
\end{lem}

We shall use this convexity in the proof of the following fact. In what follows, $L = (\Delta - \partial/\partial t)$ is the heat operator on $(\C, \sigma)$. 

\begin{lem}\label{n2}
The function $f_\alpha\circ u_t$ is a subsolution of the heat equation, that is, $L(f_\alpha\circ u_t) \geq 0$.
\end{lem}
\begin{proof}

The following  basic composition formula can be found in \cite[Proposition 2.20]{EL83}]. 
\begin{equation}\label{elcomp1}
\Delta(f\circ u_t)=tr\nabla df(du_t,du_t) + df(tr\nabla du_t):
\end{equation}

 Since $\frac{\partial}{\partial t}(f\circ u_t)=df(tr\nabla d u_t)$, we obtain  \[L(f\circ u_t)=\Delta(f\circ u_t)-\frac{\partial}{\partial t}(f\circ u_t)=tr\nabla df(d u_t,d u_t).\]
Convexity of $f$ implies that $\nabla df$ is a positive definite quadratic form  and consequently, $L(f\circ u_t)\geq 0$  as desired. 
\end{proof}

We shall also need:

\begin{lem}\label{mle}
Let $\gamma:[0,1]\rightarrow \C $ be a path and $f:(\C,\sigma)\rightarrow \mathbb{H}^3$ be a $C^2$-smooth map. Then the length $l(f(\gamma))\leq C \int_{0}^{1}(e(f)(\gamma(t)))^{1/2}dt$, for some constant $C>0$.  (Here $e(f)$ is the energy density of $f$, see Definition \ref{defn:harm}.) 
\end{lem}
\begin{proof}
Let $\gamma^i$'s be components of $\gamma$ in local coordinates around $\gamma(t)$ and $h$ be the hyperbolic metric on $\mathbb{H}^3$ where we use the upper half-space model.
\begin{align*}
l(f(\gamma))&=\int_{0}^{1}|(f\circ\gamma)'(t)|dt =\int_{0}^{1}\left|df_{\gamma(t)}\left(\sum_{i=1}^2\dot{\gamma}^i\frac{\partial}{\partial x^i}|_{\gamma(t)}\right)\right|dt\\
&=\int_{0}^{1}\left(\sum_{k,l=1}^3\sum_{i,j=1}^2\dot{\gamma}^i\dot{\gamma}^j\frac{\partial f^k}{\partial x^i}\frac{\partial f^l}{\partial x^j}h_{kl}(f(\gamma(t)))\right)^{1/2}dt\\
&=\int_{0}^{1}\left(\sum_{k=1}^3\sum_{i,j=1}^2\dot{\gamma}^i\dot{\gamma}^j\frac{\partial f^k}{\partial x^i}\frac{\partial f^k}{\partial x^j}h_{kk}(f(\gamma(t)))\right)^{1/2}dt\quad(\because h_{kl}(x,y,z)=\delta_{kl}/z^2)
\end{align*}
Put $A_i=\dot{\gamma}^i\frac{\partial f^k}{\partial x^i}$. Then applying the general inequality \[\sum_{i,j=1}^2A_iA_j\leq 2\sum_{i=1}^2A_i^2,\] we obtain
\begin{align*}
&l(f(\gamma))
\leq 2\int_{0}^{1}\left(\sum_{k=1}^3\sum_{i=1}^2(\dot{\gamma}^i)^2\left(\frac{\partial f^k}{\partial x^i}\right)^2h_{kk}(f(\gamma(t)))\right)^{1/2}dt\\
\leq& 2\int_{0}^{1}\left(\sum_{k=1}^3\sum_{i=1}^2(\dot{\gamma}^i)^2\sigma_{ii}(\gamma(t))\sigma^{ii}(\gamma(t))\left(\frac{\partial f^k}{\partial x^i}\right)^2h_{kk}(f(\gamma(t)))\right)^{1/2}dt\quad(\because \sigma_{11}\sigma^{11}=1=\sigma_{22}\sigma^{22})\\
\leq& 2P^{1/2}\int_{0}^{1}(e(f)(\gamma(t)))^{1/2}dt,
\end{align*}
where $P=\underset{i\in \{1,2\}}{\max} P_i$ with $P_i=\underset{t\in[0,1]}{\sup}(\dot{\gamma}^i)^2(\sigma_{ii}(\gamma(t))$.
\end{proof}

\vspace{.05in}

Note the following parabolic maximum principle for non-compact manifolds  (see, for example, \cite[Lemma 2.1]{MengWang} where it is attributed to Li).

\begin{prop}\label{mpr}
Let $(M,g)$ be a complete Riemannian manifold.
If $G(x,t)$ is a weak subsolution of the heat equation defined on $M\times [0,T]$ and $G(x,0)\leq 0$ for any $x\in M$, then $G(x,t)\leq 0$ for $(x,t)\in M\times [0,T]$ provided $\int_{0}^{T}\int_{M} e^{-cr^2(x)}G^2(x,t)dxdt<\infty$ for  some $c>0.$ 
\end{prop}

We shall apply the above maximum principle in the proof of the following:

\begin{lem}\label{chs}
There is a $d>0$ such that for all $t>0$, the image of $u_t$  is contained in $N_d(Q)$.
\end{lem}
\begin{proof}
Recall that whenever we have two subsolutions for the heat operator $L$, then so is their max, and hence  $$G(x,t):=\text{max}_{\alpha\in \Lambda}f_{\alpha}(u_t(x))$$ is a weak subsolution for the heat operator $L$ by the previous Lemma \ref{n2}.  Note that the distance satisfies 
\begin{align}\label{bound1}
\nonumber &d(u_t(x),u_0(x)) \leq \int_{0}^{t}\left|\frac{\partial u}{\partial s}\right|(x,s)ds =\int_{0}^{t}|\tau(u(x,s))|ds\\\nonumber
&\leq \int_{0}^{t}b(x,s)ds =\int_{0}^{t}\left[\int_{\mathbb{C}}H(x,y,s)|\tau(u_0)|^2(y)dy\right]^{1/2}ds\end{align} 
where we have used Theorem \ref{wang} for the inequality in the second line.

If $\lVert \tau(u_0)(x)\rVert_\infty=C$, then using the fact that $ \int_{\mathbb{C}}H(x,y,s)dy\leq1$, we obtain \[d(u_t(x),u_0(x))\leq Ct.\]
 By construction the image of $u_0$ intersects a compact set $\widehat{K}$ such that $u_0(x_0)\in \widehat{K}$. 
 Now 
\begin{align*}
d(u_t(x),H_{\alpha})=& \inf_{p\in H_\alpha} d(u_t(x),p) \leq d(u_t(x),u_0(x))+d(u_0(x),u_0(x_0))+ \inf_{p\in H_\alpha} d(u_0(x_0),p)\\
	&\leq Ct+ C_1|x-x_0|e(u_0)+C_2,
\end{align*} where in the last inequality we have used the Lemma \ref{mle}. By our construction, the image of $u_0$ is contained in the neighborhood $N_d(Q)$ for some $d>0$. In fact, from our construction in \S3.1.3, the image of each vertical half-plane $C_i$ in $\C$ maps \textit{into} the convex hull $Q$. Since the union of these planes covers $\C$ except a compact set $K$, we allow a distance $d$ for the extension to $K$. 
Therefore, $G(x,0)\leq d$. To apply the parabolic maximum principle (Proposition \ref{mpr}), it remains to prove the integral condition
\begin{align*}
&\int_{0}^{T}\int_{M} e^{-cr^2(x)}G^2(x,t)dxdt\\
&\leq \int_{0}^{T}\int_{M}dC^2 e^{-cr^2(x)}t^2dxdt+dC_1e(u_0)\int_{0}^{T}\int_{M} e^{-cr^2(x)}|x-x_0|dxdt\\
&+dC_2\int_{0}^{T}\int_{M} e^{-cr^2(x)}G^2(x,t)dxdt\\
&< \infty
\end{align*}
Therefore,  Proposition \ref{mpr} applies, and we have $G(x,t)\leq d$ for all $t>0$, proving the result.
\end{proof}

We note the following corollary :

\begin{cor}\label{compact}
There is a compact set $\widehat{K}\subset \mathbb{H}^3$  such that the image of $u_t$ intersects $\widehat{K}$ for all $t\in[0,\infty)$.
\end{cor}
\begin{proof}
Let $\widehat{K}$ be the compact core of the neighborhood  $N_d(Q)$ of the ideal polyhedron $Q$ in the statement of Lemma \ref{chs}, that is obtained by removing the cuspidal ends (intersections of horoballs at each of the ideal vertices). Note that $N_d(Q) \setminus \widehat{K}$ is disconnected, and has exactly $n$ components. 
Using Corollary \ref{tend0} we know that for each $t>0$, the maps $u_t$  and $u_0$ are asymptotic to the same ideal vertices. If there exists $t_0$ such that the image of $u_{t_0}$ does not intersect $\widehat{K}$, then by connectedness, the image of $u_{t_0}$ will miss some end of $Q$, which is a contradiction. This completes the proof.
\end{proof}

\subsection{Convergence to $u_\infty$ with desired asymptotics}
In the previous subsection, we were able to show that the harmonic map heat flow $u_t:\C \to \HH^3$ starting with the initial map $u_0$ constructed in \S3.1 has an image that intersects a fixed compact set $\widehat{K}$ at each time $t\geq 0$ (see Corollary \ref{compact}). However, since the domain is non-compact, this alone does not imply one-point convergence, namely that there exists a point $p\in \C$ such that $u_t(p)$ convergences as $t\to \infty$ (perhaps along a subsequence). Such a one-point convergence would have immediately implied the convergence of the flow to a limiting harmonic map (see, for example, \cite[Theorem 4.3]{LiTam}). 

We get around this difficulty by first defining an auxiliary map $\Xi:\C \to \mathbb{H}^3$ that is $C^0$, but is \textit{piecewise-harmonic}; this is the \textit{pleated plane map} that we define in \S3.5.1. Next, in \S3.5.2, we show that the distance function $\Psi(x,t) = d(u_t(x), \Xi(x))$ is a weak subsolution of the heat equation; we then obtain the desired uniform distance bound by an application of the Parabolic Maximum Principle (Proposition \ref{mpr}).

\subsubsection{The pleated plane map}

Recall that there exists a planar polygon $P_0$ in $\mathbb{H}^3$ such that the twisted ideal polygon $P$ is obtained by bending $P_0$ along a collection of disjoint diagonals (see Lemma \ref{bend}). Here, ``planar" means that $P_0$ is contained in a totally geodesic copy of $\HH^2$, that we denote by $H$.

\medskip 

 Let $\mathcal{C} = \{d_1, d_2,\ldots, d_{k-1}\}$ be the  collection  of diagonals of $P_0$ bending along which results in the twisted ideal polygon $P$. In the following definition, $R_0$ will be subset of the totally-geodesic hyperbolic plane $H \subset \HH^3$ bounded by $P_0$; the geodesics in $\mathcal{C}$ then partition $R_0$ into subsets $R_1, R_2,\ldots R_k$ (for some $k$), namely we have $R_0 \setminus \bigcup\limits_{\gamma \in \mathcal{C}} \gamma= \bigsqcup\limits_{i=1}^k R_i$. 
 
 From the proof of Lemma \ref{bend}, there is a \textit{pleated polygon} in $\mathbb{H}^3$, which bounds the twisted ideal polygon $P$ and is a piecewise-totally geodesic embedding  $\Phi: R_0 \to \mathbb{H}^3$ described more precisely as follows:

\begin{defn}[Pleated plane $\Phi$]\label{mapPhi}
 Let $e_1, e_2,\ldots e_{k-1} \in \pslc$ be elliptic isometries such that $\text{Axis}(e_i) = d_i \in \mathcal{C}$ for each $1\leq i\leq k-1$. Here, the rotation angle of $e_i$  is $\theta_i$, the argument of the complex cross-ratio corresponding to the diagonal $d_i$ when considering such cross-ratio parameters for the twisted ideal polygon $P$ (\textit{c.f.} the proof of Lemma \ref{bend}). Define a \textit{bending cocycle} $B:R_0 \times R_0 \to \pslc$ as follows: for any  pair of points $x, y\in R_0$, consider the directed geodesic segment from $x$ to $y$ and let $d_{k_1}, d_{k_2},\ldots d_{k_m}$ be the ordered set of  diagonals in $\mathcal{C}$ that the segment crosses. Then define $$B(x,y) = e_1 \circ e_{k_2} \circ \cdots \circ e_{k_m}$$ where recall $e_i$ is the elliptic isometry corresponding to the diagonal $d_i$.  (This is a cocycle in the appropriate sense -- see the discussion in \cite[\S5.3]{Dumas}.) Then the pleated polygon $\Phi:R_0 \to \mathbb{H}^3$ is defined as follows: fix a basepoint $x_0 \in R_0$ and let $\Phi(x) = B(x_0,x)\Phi_0(x)$, where $\Phi_0 :R_0 \to \mathbb{H}^3$ is the totally-geodesic embedding that one started with.  (In Figure \ref{fig2}, the pleated plane is shown shaded on the right.)
 \end{defn}

  \noindent \textit{Remark.} The notion of a ``pleated plane", or more precisely, ``pleated surface" in a hyperbolic $3$-manifold, was introduced by Thurston in \cite[Chapter 8]{Thurston}.

\vspace{.05in}

\noindent We can now define:

\begin{defn}[The map $\Xi$]\label{mapP} The \textit{pleated plane map} $\Xi:\C\to \HH^3$ is a continuous map defined as follows: Choose a harmonic map $h:\C \to \HH^2$ that is a diffeomorphism to the region $R_0$ bounded by $P_0$ (\textit{c.f.} Theorem \ref{httw}), and let $\Xi = \Phi\circ h$ where $\Phi:R_0 \to \HH^3$ is the pleated plane defined above. 
\end{defn} 


  \medskip

Since a harmonic map post-composed by an isometry is again harmonic, the following property is immediate from the above construction:

\begin{lem}\label{Pharm} The $C^0$-map $\Xi$ is harmonic away from preimages of the diagonals $h^{-1}(\mathcal{C}) = \{h^{-1}(d) \ \vert\ d\in \mathcal{C}\}$ in $\mathbb{C}$. 
\end{lem}

\noindent We also observe the following:

\begin{lem}\label{Pb} The distance function $\Psi_0(x) = d(u_0(x), \Xi(x))$ is uniformly bounded, i.e.\ there exists an $M>0$ such that $\Psi_0(x) \leq M$ for all $x\in \C$.
\end{lem}

\begin{proof}
Recall from the end of \S3.1.1 that we had chosen a harmonic map $h:\C \to \mathbb{H}^2$ asymptotic to the ``straightened" planar ideal polygon $P_0$. 
We shall use this choice of harmonic map $h$ in Definition \ref{mapP}.

Recall that the initial map  $u_0:\C \to \HH^3$ is obtained by modifying the map $h:\C \to \HH^2$ as described in \S3.1.3. Briefly, one can think of the target hyperbolic plane for $h$ as a totally-geodesic plane in $\HH^3$, and the map $u_0$ is constructed by post-composing restrictions of $h$ with maps that twist about the sides of $P_0$ (see Figure \ref{fig4}). These twist maps are either elliptic rotations, or maps that interpolate between different rotations --\textit{c.f.} \eqref{h3}.

Moreover, from our constructions the pleated plane map $\Xi:\C \to \HH^3$ is also obtained by modifying $h:\C \to \HH^2$. This modification is simpler than the one that converts $h$ to $u_0$, since it only involves post-composing with elliptic rotations (see Definition \ref{mapPhi}).

Note that in both the modifications  of  the maps involved in the post-compositions  take the geodesic sides of $P_0$ to the geodesic lines in $\HH^3$ that are the sides of the twisted ideal polygon $P$. Moreover, distances along each geodesic side of $P_0$, are preserved by the modification maps.  In addition, the modifications map the cusp regions bounded by $P_0$  to a bounded neighborhood of the cusps of $P$. This is because each cusp $\widehat{C}_i$ of $P$ is at most a bounded distance away from the corresponding end the pleated plane bounded by $P$, which is piecewise totally-geodesic and possibly bent along geodesic lines exiting that end.  Thus, $u_0$ and $\Xi$ are a uniformly bounded distance from each other.
\end{proof}

\subsubsection{Distance bound and convergence of the flow}

We note the following general fact concerning the distance function between two solutions of the harmonic map heat flow:

\begin{lem}[\cite{SchoenYau79}]\label{sy}
Let $(M,g)$ and $(N,h)$ be two Riemannian manifolds such that $N$ is simply connected non-positively curved manifold. Assume that $\Omega \subset M$ is an open subset and $v,w:\Omega \times [0,\infty)\rightarrow N$ are two solutions of the harmonic map heat flow \eqref{heat1}. Then \[\left(\Delta-\frac{\partial}{\partial t}\right)d(v,w)\geq 0\]
where $d(\cdot, \cdot)$ is the distance function with respect to the metric on $N$.
\end{lem}
\begin{proof}
Let $h:\Omega \times [0,\infty) \to N \times N$ be defined by $h(x,t)=(v(x,t),w(x,t))$. 
We then have 
\begin{align}\label{root2}
\nonumber\left(\Delta-\frac{\partial}{\partial t}\right)d(v,w)&=\Delta(d\circ h)-\frac{\partial (d\circ h)}{\partial t}\\\nonumber
&\geq\langle \nabla d,\tau(h)\rangle -\left\langle \nabla d,\frac{\partial h}{\partial t}\right\rangle\\
&= \left\langle \nabla d,\left(\tau(v)-\frac{\partial v}{\partial t}, \tau(w)-\frac{\partial w}{\partial t}\right)\right\rangle = 0 
\end{align}
where the first inequality uses the chain-rule  (see for example \cite[Equation (2.1)]{HardtWolf}) and the convexity of the distance function $d$ on $N$, and the final equality uses the fact that $v$ and $w$ are solutions of \eqref{heat1}.
\end{proof}

Note that since a harmonic map $u:M \to N$ can be thought of as a \textit{stationary} solution of the harmonic map heat flow, we have the following corollary:

\begin{cor} If $v:\Omega \times [0, \infty) \to N$ is a solution of the harmonic map heat flow, and $u:\Omega \to N$ is harmonic, then the distance function $d(v(x,t), u(x))$ is a subsolution of the heat equation, that is, $\left(\Delta-\frac{\partial}{\partial t}\right) d(v(x,t), u(x)) \geq 0$.
\end{cor}

In our setting, consider the distance function from solution $u_t$ of the harmonic map heat flow (with initial map $u_0$ constructed in \S3.1) to the map $\Xi$ constructed in the previous subsection (see Definition \ref{mapP}). Since by construction, $\Xi$ is $C^0$, and harmonic away from a collection of real-analytic arcs in $\C$ (see Lemma \ref{Pharm}), the computation \eqref{root2} holds away from them and we obtain:

\begin{lem}\label{weak} The distance function $\Psi(x,t) = d(u_t(x), \Xi(x))$ is a weak subsolution of the heat equation. 
\end{lem} 

Since the Parabolic Maximum Principle (Proposition \ref{mpr}) holds for \textit{weak} subsolutions of the heat equation, we then obtain:

\begin{cor}\label{Psib}
    There is a constant $M>0$ such that $\Psi(x,t)\leq M$ for all $(x,t) \in \C \times [0,\infty)$.
 \end{cor}
 \begin{proof}
By Lemma \ref{Pb} there exists a constant $M>0$ such that $\Psi(x,0)\leq M$ for all $x\in \C$. To apply the parabolic maximum principle (Proposition \ref{mpr}) we need to check the integral condition. 

Note that for any $(x,t) \in \C \times [0,\infty)$, we have the bound \[\Psi(x,t)\leq d(\Xi(x),u_0(x))+d(u_0(x),u_t(x))\leq M+Dt,\] for some constants $D>0$ that is independent of $t$. Here the linear bound in the second term on the RHS follows from \eqref{distb} and the fact that the tension field is uniformly bounded (see equations \eqref{bbound} and \eqref{ico11}). Hence we obtain that for any $T>0$, \[\int_{0}^{T}\int_{\mathbb{C}} e^{-r^2(x)}\Psi^2(x,t)dxdt\leq\int_{0}^{T}\int_{\mathbb{C}} e^{-r^2(x)}(C+Dt)^2dxdt<\infty\]
and Theorem \ref{mpr} applies. Thus we  conclude  $\Psi(x,t) = d(\Xi(x),u_t(x))\leq M$ where $M$ for all $x\in \C$ and $t\in \mathbb{R}^+$. 
 \end{proof}

We can now prove the convergence of the flow:

\begin{prop}\label{rtc}
The harmonic map heat flow $u_t$ converges uniformly on compact sets to a harmonic map  $u_\infty$ as $t\to \infty$. 
\end{prop}
\begin{proof}
 Corollary \ref{Psib} implies that for any $x_0\in \C$, $u_t(x_0)$ is contained in a bounded set of $\HH^3$ as $t \to \infty$. We thus obtain a subsequence of times $t_n \to \infty$ such that $u_{t_n}(x_0) \to p$, i.e. we obtain one-point convergence.  For ease of notation, we shall denote $u_n := u_{t_n}$. 

Recall that the energy density of ${u}_n$ is uniformly bounded and independent of $n$ by Lemma \ref{edb}. Since the energy density is the norm of the gradient, this is equivalent to a uniform bound on first derivatives (in the space variable) of $u_n$. Standard bootstrapping techniques applying Schauder estimates  for  solutions of a parabolic PDE (see, for example, \cite[Appendix A.2(e)]{Nishikawa02}) then implies uniform bounds on higher order derivatives as well. Applying Arzela-Ascoli's theorem we conclude that ${u}_n$ converges to a $C^2$-smooth map ${u}_{\infty}$ uniformly on compact subsets.

By Lemma \ref{wuie}, $\lVert \tau (u_n) \rVert_\infty\to 0$ uniformly as $n\to \infty$. Thus, $\lVert \tau({u}_\infty) \rVert_\infty =0$ and ${u}_{\infty}$ is harmonic, as desired. Note that by the previous Corollary, $u_\infty$ is a uniformly bounded distance away from $u_0$.

We know that  there is a unique harmonic map that is a bounded distance away from $u_0$:  Indeed, Lemma \ref{sy} implies that the distance function between two harmonic maps from $\C$ to $\HH^3$ is a  subharmonic function on $\C$, and hence constant. Hence, if two such harmonic maps are a \textit{bounded} distance apart, they are a constant distance apart, and one can argue as in \cite[Lemma 3.11]{Sagman} that the constant is in fact zero. 

From this uniqueness, it is then not hard to conclude that for  \textit{any} sequence of times $s_n \to \infty$, the maps $u_{s_n}$ converges uniformly on compact sets to $u_\infty$, i.e.\ the original harmonic map heat flow converges to $u_\infty$ as $t\to \infty$. 
\end{proof}

We conclude by observing that the limiting harmonic map $u_\infty$ is indeed the desired map:

\begin{lem} The harmonic map $u_\infty:\C\to \HH^3$ is asymptotic to the twisted ideal polygon $P$, and has a polynomial Hopf differential.
\end{lem}
\begin{proof}
We first show that the limiting harmonic map $u_\infty$ has a polynomial Hopf differential $q$ (of some degree). Indeed, since $u_\infty$ is harmonic we know that $\text{Hopf}(u_\infty) = q(z) dz^2$  where $q$ is an entire function.

A computation  (see, for example, \cite[\S2.2]{WolfKnox}) shows that the norm $$\lVert q \rVert^2 = q^2(z)/\sigma^2(z) = \mathcal{H} \mathcal{L}$$ where $\sigma$ is the conformal metric on the domain complex plane (see \S3.2), and $\mathcal{H} = \lVert \partial u_\infty \rVert^2$ and $\mathcal{L} = \lVert \bar{\partial}  u_\infty \rVert^2 $ are the holomorphic and antiholomorphic energy densities of $u_\infty$.

Since the energy density $e(u_\infty)= \mathcal{H} + \mathcal{L}$ it follows that $$\lvert q(z)\rvert \leq \sigma(z) e(u_\infty)(z).$$

Recall that the energy density  $e({u}_t)$ is uniformly bounded by Lemma \ref{edb}.
Moreover, from \S3.2,  the conformal factor $\sigma$ has at most polynomial growth, since it is a smoothening of the $\phi$-metric on $\C$, where that $\phi$ is a polynomial Hopf differential  of a harmonic map $h:\C \to \HH^3$. These imply that $q(z)$ has at most polynomial growth, and is thus a polynomial, as desired. By Proposition \ref{myprop}, the map $u_\infty$ is asymptotic to a twisted ideal polygon in $\HH^3$. 

It remains to show that this must be the given  twisted ideal polygon $P$. 
By construction, the initial map $u_0$ is asymptotic to $P$. By the uniform distance bound from $\Xi$ along the flow -- see Corollary \ref{Psib} --  and the convergence $u_t \to u_\infty$,  the  map $u_\infty$ is a bounded distance from $u_0$. Thus, $u_\infty$ hence asymptotic to the twisted ideal polygon with the same ideal vertices as $P$, which is exactly $P$.
\end{proof}

This completes the proof of Theorem \ref{main}.

\section{Uniqueness with prescribed principal part}

In this  final section, we shall characterize the non-uniqueness of the harmonic maps obtained in Theorem \ref{main0}. We first 
show that our construction in fact yields \textit{infinitely many} harmonic maps asymptotic to the given twisted ideal polygon.  Then, we shall prove a uniqueness statement when one prescribes a principal part of the Hopf differential (see Definition \ref{defn:princ}).

\subsection{Non-uniqueness}\label{repeateq}
To show non-uniqueness, observe that if one starts with another initial map $u_0^2$ which has the same asymptotics as $u_0^1$ and is at a bounded distance from $u_0^1$, then the harmonic map heat flow converges to the same harmonic map. Indeed, if we call the limiting harmonic maps $u_\infty^1$ and $u_{\infty}^2$ for the two flows, then the function $d_{\mathbb{H}^3}(u_{\infty}^1(x),u_{\infty}^2(x))$ is a uniformly bounded subharmonic function on $\C$ and hence constant, and the constant is identically zero by the same argument as in \cite[Lemma 3.11]{Sagman}.

We therefore have to construct initial maps at an \textit{unbounded distance} from each other.

\vspace{.05in} 

Recall that in the beginning of the construction of the initial map (\S3.1.1) we started with a harmonic map $h:\C \to \HH^2$ that is asymptotic to a planar ideal polygon $P_0$. There are different choices of such a harmonic map -- this fact is implicit in \cite{HTTW95} just by comparing dimensions; for a description of the space of such maps, see \cite{Gupta21}. Any such a pair of such distinct harmonic maps will necessarily be an unbounded distance apart, by the argument above. It remains to show:

\begin{lem}\label{nunnes}
In our construction described in \S3.1,  two distinct harmonic maps $h, h^\prime:\C \to \HH^2$ asymptotic to the planar ideal polygon $P_0$ determine two initial maps $u_0,u_0^\prime: \C \to \HH^3$ that are an unbounded distance apart. 
\end{lem}
\begin{proof}
 We briefly recall how $h,h'$ are related to $u_0,u_0'$ respectively.  Recall from \S3.1.2 that the domain complex plane has a compact set $K$ whose complement is a chain $\{C_1, H_1, C_2, H_2,\ldots C_n, H_n\}$ of vertical and horizontal half-planes, each successive pair overlapping on a quarter-plane. For any $1 \leq i\leq n$, the initial map $u_0$ is defined to equal $h$ (up to post-composition by an isometry) on the quarter-plane $C_i\cap H_i$, and is defined to be $E\circ h$ in $C_{i+1} \cap H_{i}$, where $E$ is an elliptic rotation by some angle $\theta_0$ (that depends on $i$). On the half-infinite strip $[a_i,a_{i+1}] \times [0,\infty)$ between $C_i \cap H_i $ and $C_{i+1} \cap H_{i}$, $u_0$ is defined to be an interpolating map between $h$ and $E\circ h$. The interpolation is obtained by modifying $h$ by a post-composition by a rotation around the axis of $E$, where the rotation angle smoothly increases from $0$ to $\theta_0$ in the interval $[a_i,a_{i+1}]$ (see \S3.1.3 for details).
 
 Note that if $d(h|_{C_i}(x),h^\prime|_{C_i}(x))$ and $d(h|_{H_i}(x), h^\prime|_{H_i}(x))$  are both uniformly bounded for all $i$, then $h$ and $h^\prime$ are bounded distance apart as $\cup_{i=1}^n(C_i\cup H_i)=\mathbb{C}\setminus K$. We can therefore assume that there is an $i$ such that one of the two distance functions is unbounded. If $d(h|_{C_i}(x),h^\prime|_{C_i}(x))$ is unbounded,  then so is $d(u_0|_{C_i}(x),u_0^\prime|_{C_i}(x))$ since on $C_i$, the initial maps $u_0$ and $u_0^\prime$ are $h$ and $h^\prime$ post-composed by the same isometric embedding of $C_i$ into the cusped region $\widehat{C}_i \subset \HH^3$. 
 If $d(h|_{H_i}(x),h^\prime|_{H_i}(x)) \to \infty $ as $k\to \infty$, we shall now argue  that $d(u_0|_{H_i}(x),u_0^\prime|_{H_i}(x))$ is unbounded as well: consider  a sequence $\{x_k\}_{k\geq }$ in $H_i$ such that $d(h(x_k),h'(x_k))$ is unbounded. This sequence is necessarily diverging in $H_i$, and by Proposition \ref{est} its image is uniformly close to a geodesic line $\gamma_i \subset \HH^3$ (namely,the $i$-th side of the twisted ideal polygon $P$). Since for each $x\in H_i$ the points $u_0(x)$ and $u_0^\prime(x)$ are obtained by rotating $h(x)$ and $h^\prime(x)$ respectively, by some angle with axis $\gamma_i$, we have that $d(u_0(x), h(x))$ and $d(u^\prime_0(x), h^\prime(x))$ are both bounded by $2d$ where $d$ is the distance to $\gamma_i$.  The triangle inequality then yields\[d(h(x_k),h'(x_k))\leq 2d+d(u_0(x_k),u_0'(x_k))+2d\]
 and since the left hand side is unbounded, so is the distance function $d(u_0(x),u'_0(x))$. This completes the proof. 
\end{proof}

\subsection{Principal parts and uniqueness}
We shall conclude by characterizing the non-uniqueness in Theorem \ref{main0}, in terms of the notion of the principal part (Definition \ref{defn:princ}). Throughout this subsection,  $P_0$ shall be an ideal polygon with $n\geq 3$ sides in a totally-geodesic plane $H \subset \HH^3$, and $P$ is a twisted ideal polygon obtained in $\HH^3$ by ``bending" P along a collection of disjoint diagonals (\textit{c.f.} Lemma \ref{bend}). Moreover, by a normalization (post-composition with an isometry) we can assume that one of the cusps of $P$ and $P_0$ has the same ideal boundary point, and we number the cusps in both polygons in a cyclic order starting with this.

\vspace{.05in}

We start with the following lemma:

\begin{lem}\label{lem:princ}
Let $\mathbb{P}$ and $\mathbb{P}^\prime$ be the principal parts of the Hopf differentials of the harmonic maps $h,u:\C \to \HH^3$ that are asymptotic to $P_0$ and $P$ respectively, and are normalized such that a fixed direction in $\C$ is asymptotic to the same ideal point. Then there is a uniform bound
\begin{equation}\label{hub}
    \lvert d(h(x), p_0) - d(u(x), p_0) \rvert < D \text{ for all }x\in\C
\end{equation}
where $p_0$ is a choice of a basepoint in $\HH^3$ and $D$ is independent of $x$,  if and only $\mathbb{P} =\mathbb{P}^\prime$.
\end{lem}

\begin{proof}
In one direction, assume that the two principal parts are equal. Then the argument in the proof of \cite[Proposition 3.9]{Gupta21} can be adapted here.  Namely, consider an exhaustion of $\C$ with polygons $\{G_k\}_{k\geq 1}$ having $n$ alternating horizontal and vertical sides. By the equality of principal parts and \cite[Lemma 3.7]{Gupta21}, the lengths of these sides in both Hopf differential metrics differ by a uniformly bounded constant.  (We can choose this exhaustion such that the lengths of the sides of $G_k$ are all $L_k$ up to a uniform additive error, and  $L_k \to \infty$ as $k\to \infty$). By the normalization,  both maps will take the $i$-th vertical side into the $i$-th cusp of $P$. Moreover, by Proposition \ref{est} and the remark following Proposition \ref{myprop}, and an argument as in the Claim in the proof of \cite[Proposition 3.9]{Gupta21}, the $i$-th side will be a distance $L_k + O(1)$ into the $i$-th cusp, for both maps. Choosing a basepoint $p_0 \in \HH^3$, \eqref{hub} follows, namely the difference of distances from $p_0$ is uniformly bounded. 

Conversely, if $\mathbb{P} \neq \mathbb{P}^\prime$, then by \cite[Lemma 3.8]{Gupta21} there is a sequence of points $z_i$ diverging in $\C$, such that the horizontal distances of $z_i$ from a fixed basepoint $z_0\in \C$ with respect to the two Hopf differential metrics have an unbounded difference. By passing to a subsequence, one can assume that this sequence of points lie in the $i$-th vertical half-plane in both metrics. Then by the same estimates as above, the distance into the $i$-th cusp of $P_0$ that $h$ maps the $i$-th vertical side into,  and the distance  into the $i$-th cusp of $P$ of the image of same side under $u$, have an unbounded difference.  \end{proof}

As a corollary, we then obtain:

\begin{prop}\label{prop:princ} Given any principal part $\mathbb{P}$ compatible with the planar polygon $P_0$, there exists a unique harmonic map $u:\C\to \HH^3$ that is asymptotic to the twisted ideal polygon $P$.

\vspace{.05in} 

(Here ``compatible" is in the sense of \cite[Definition 2.30]{Gupta21}, namely that when $n$ is even, the real part of the residue of $\mathbb{P}$ equals the metric residue of the $P_0$, as defined in \cite[Definition 2.9]{Gupta21}.)

\end{prop}

\begin{proof}
    First, we prove the existence statement. 
    As observed at the end of \S3.1.1, by \cite[Proposition 3.12]{Gupta21} one can choose the harmonic map $h:\C \to H \subset \HH^3$ to have principal part $\mathbb{P}$. From the construction of the initial map $u_0$ by modifying $h$, it follows that the distance functions $d(h(x), p_0)$ and  $d(u_0(x), p_0)$ are uniformly bounded from each other. From Corollary \ref{Psib}, it follows that there is also a uniform distance bound $d(u_0(x), u_\infty(x)) \leq M$, where $u = u_\infty$ is the limiting harmonic map that the harmonic map heat flow starting with $u_0$ converges to (by Theorem \ref{main}). 
    Putting these together, it follows that \eqref{hub} holds. Hence, by Lemma \ref{lem:princ}, the principal part of the Hopf differential of $u$ equals $\mathbb{P}$.

    The uniqueness follows from the same argument as in the proof of \cite[Proposition 3.9]{Gupta21}: if there are two harmonic maps $u_1,u_2:\C \to \HH^3$ asymptotic to the same twisted ideal polygon $P$, with Hopf differentials  having the same principal part $\mathbb{P}$, then the distance estimates for the map (see the remark following Proposition \ref{myprop}) and comparability of the flat metrics (see \cite[Lemma 3.7]{Gupta21}) implies that there is a uniform distance bound between the two maps. (Here, it is crucial that they are both asymptotic to the same twisted ideal polygon.) As observed in \S4.1, this implies that $u_1=u_2$. \end{proof}

  \noindent \textit{Remark.} Note that the straightened planar polygon $P_0$ given by Lemma \ref{bend} is not unique, however the notion of compatibility above does not depend on the choice of such a straightening. Indeed, one can define the ``metric residue" of the twisted ideal polygon $P$ similar to \cite[Definition 2.9]{Gupta21} by truncating the sides by a a choice of a horosphere at each ideal vertex of $P$.

\bibliography{thesis_ref}{}
\bibliographystyle{alpha}

\end{document}